\documentclass[11.5pt]{amsart}
\usepackage{graphicx,amssymb,amsmath,amsthm}
\usepackage{comment}
\usepackage{enumerate}
\usepackage{dsfont}
 \numberwithin{equation}{section}
\usepackage{mathrsfs}
\usepackage{epsfig}
\usepackage{lscape}
\usepackage{subfigure}
\usepackage{graphicx}
\usepackage{epstopdf}
\usepackage{color}
\usepackage{caption}
\usepackage{algorithm}
\usepackage{algpseudocode}
\usepackage{bm}
\usepackage{graphicx}
\usepackage{subfig}

\allowdisplaybreaks[4]

\newcommand{\abs}[1]{\lvert#1\rvert}

\newcommand{\argmin}[1]{\mathop{\rm argmin}\limits_{#1}}

\newcommand{\A}{{\mathcal A}}

\newcommand{\R}{{\mathbb R}}
\newcommand{\cA}{{\mathcal A}}
\newcommand{\vA}{{\boldsymbol A}}

\newcommand{\vx}{{\boldsymbol x}}

\newcommand{\vy}{{\boldsymbol y}}
\newcommand{\vu}{{\boldsymbol u}}
\newcommand{\vg}{{\boldsymbol g}}
\newcommand{\vv}{{\boldsymbol v}}
\newcommand{\vz}{{\boldsymbol z}}
\newcommand{\vb}{{\boldsymbol b}}
\newcommand{\ve}{{\boldsymbol \eta}}

\newcommand{\vZ}{{\boldsymbol Z}}
\newcommand{\vX}{{\boldsymbol X}}

\newcommand{\va}{{\boldsymbol a}}
\newcommand{\vw}{{\boldsymbol w}}

\newcommand{\vH}{{\boldsymbol H}}
\newcommand{\innerp}[1]{\langle #1 \rangle}

\newcommand{\F}{{\mathbb F}}

\newcommand{\rank}{{\rm rank}}

\renewcommand{\omega}{\ve}

\newcommand{\RNum}[1]{\uppercase\expandafter{\romannumeral #1\relax}}

\newtheorem{definition}{Definition}[section]

\newtheorem{theorem}[definition]{Theorem}
\newtheorem{lemma}[definition]{Lemma}

\newtheorem{remark}[definition]{Remark}

\date{}

\begin{document}
\baselineskip 18pt
\bibliographystyle{plain}
\title{The performance of the amplitude-based model for complex phase retrieval}
\author{Yu Xia}
\address{School of Mathematics,
Hangzhou Normal University, Hangzhou 311121, China}
 \email{ yxia@hznu.edu.cn}
\author{Zhiqiang Xu}
\thanks{Yu Xia is supported by NSFC grant (12271133, U21A20426,11901143), Zhiqiang Xu is supported  by
the National Science Fund for Distinguished Young Scholars (12025108) and NSFC (12021001, 12288201). }
\address{LSEC, ICMSEC, Academy of Mathematics and Systems Science, Chinese Academy of Sciences, Beijing 100190, China;\newline
School of Mathematical Sciences, University of Chinese Academy of Sciences, Beijing 100049, China}
\email{xuzq@lsec.cc.ac.cn}
\begin{abstract}
The paper aims to study the performance of the amplitude-based model \newline $\widehat{\vx} \in \argmin{\vx\in \mathbb{C}^d}\sum_{j=1}^m\left(|\langle \va_j,\vx\rangle|-b_j\right)^2$, where $b_j:=|\langle \va_j,\vx_0\rangle|+\eta_j$ and $\vx_0\in \mathbb{C}^d$ is a target signal. The model is raised in phase retrieval as well as in absolute value rectification neural networks. Many efficient algorithms have been developed to solve it in the past decades.  {However, there are very few results available regarding the estimation performance in the complex case under noisy conditions.}  In this paper,  {we present a theoretical guarantee on the amplitude-based model for the noisy complex phase retrieval problem}. Specifically, we show that $\min_{\theta\in[0,2\pi)}\|\widehat{\vx}-\exp(\mathrm{i}\theta)\cdot\vx_0\|_2 \lesssim \frac{\|\ve\|_2}{\sqrt{m}}$ holds with high probability provided   the measurement vectors $\va_j\in \mathbb{C}^d,$ $j=1,\ldots,m,$ are {i.i.d.} complex sub-Gaussian random vectors and
  $m\gtrsim d$. Here $\ve=(\eta_1,\ldots,\eta_m)\in \mathbb{R}^m$ is the noise vector without any assumption on the distribution.    Furthermore, we prove that the reconstruction error is sharp. For the case where the target signal $\vx_0\in \mathbb{C}^{d}$ is sparse, we establish a similar result for the nonlinear constrained $\ell_1$ minimization model.    { To accomplish this, we leverage a strong version of restricted isometry property for an operator on the space of simultaneous low-rank and sparse matrices.}
\end{abstract}
\maketitle

\section{Introduction}

\subsection{Problem setup}
The aim of phase retrieval is to recover a  target signal from phaseless measurements.  {This problem arises in various applications, such as X-ray crystallography \cite{app1}, quantum mechanics \cite{app2}, optics \cite{Fienup1}, astronomy \cite{RP}, and many other fields. For a comprehensive understanding of recent reconstruction algorithms and the latest theoretical advancements, please refer to \cite{Unser23}.}  In this paper, we study the recovery of a signal $\vx_0\in \mathbb{F}^{d}$ $(\mathbb{F}\in \{\mathbb{R},\mathbb{C}\})$  from noise corrupted amplitude-based measurements $b_j$,  $j=1,\ldots,m$, that is,
\begin{equation}\label{eqn: amplitude_observation}
b_j:=|\langle \va_j,\vx_0\rangle|+\eta_j,\qquad j=1,\ldots,m,
\end{equation}
where $\va_j\in \mathbb{F}^n$ are the measurement vectors and $\ve:=(\eta_1,\ldots,\eta_m)^\top\in \mathbb{R}^m$ is a noise vector.
In the noiseless case, i.e. $\ve=0$, it was shown that $m\geq 2d-1$  (resp. $m\geq 4d-4$) generic measurements are sufficient to exactly recover $\vx_0\in \F^d$ up to a unimodular constant  for $\mathbb{F}=\mathbb{R}$  (resp. $\mathbb{F}=\mathbb{C}$)   \cite{generic1,generic2,generic3}.

In the noisy case, an intuitive method to estimate $\vx_0$ from the amplitude-based observations (\ref{eqn: amplitude_observation}) is to solve  the following nonlinear least square model:
 \begin{equation}\label{eqn: nonsparse_model}
 \argmin{\vx\in \mathbb{F}^d}\,\, \sum_{j=1}^m\left(|\langle \va_j,\vx\rangle|-b_j\right)^2.
 \end{equation}
Apart from phase retrieval, the model (\ref{eqn: nonsparse_model}) has also been studied in the context of shallow neural networks with absolute value rectification (see \cite{ReLU1, ReLU2}).

{{For the case where the target signal  $\vx_0$ is sparse, we can employ the following $\ell_1$-minimization model to recover $\vx_0$:
 \begin{equation}\label{eqn: sparse_model}
 \argmin{\vx\in \mathbb{F}^d} \|\vx\|_1\qquad \text{s.t.} \qquad \left(\sum_{j=1}^m\left(|\langle \va_j,\vx\rangle|-b_j\right)^2\right)^{1/2}\leq \epsilon,
 \end{equation}
 where $\epsilon$ is a parameter related to the noise level such that $\|\boldsymbol{\eta}\|_2\leq \epsilon$.}} {In practice, the noise level is determined by the system itself.   For example, assume that $\eta_i$, $i=1,\ldots, m$, are independently drawn from Gaussian distribution $\mathcal{N}(0,\sigma^2)$. Then, for $t\geq 0$, we have
\[
\mathbb{P}\{\|\boldsymbol{\eta}\|_2^2\geq (m+2\sqrt{m}\sqrt{t}+2t)\sigma^2\}\leq \exp(-t),
\]
which is derived from Chi-square concentration inequality \cite{LM}. Therefore, by choosing  $$\epsilon=\sqrt{m+2\sqrt{m\log m}+2\log m}\sigma,$$  the target signal $\vx_0$ satisfies the constrain condition
in (\ref{eqn: sparse_model})
 with a probability of at least $1-1/m$.   The estimation of the parameter $\sigma$ can be viewed as a moment estimation problem, as discussed in \cite{noise}. }

 Naturally, one is interested in assessing the quality of the recovery of $\vx_0$ by solving (\ref{eqn: nonsparse_model}) and (\ref{eqn: sparse_model}).
 For the case  where $\mathbb{F}=\mathbb{R}$,  the performances of (\ref{eqn: nonsparse_model}) and (\ref{eqn: sparse_model}) were studied in \cite{Real_nonlinear_LS} and \cite{sparse2}, respectively. It was shown that if $\va_j,$ $j=1,\ldots,m,$ are independent real Gaussian random vectors, the solution $\widehat{\vx}$ to (\ref{eqn: nonsparse_model})  satisfies $\min\{\|\widehat{\vx}-\vx_0\|_2, \|\widehat{\vx}+\vx_0\|_2\}\lesssim\|\ve\|_2/\sqrt{m}$ with high probability provided that $m \gtrsim d$.  
{{One also shows that the reconstruction error is sharp \cite{Real_nonlinear_LS}, that is, $\min\{\|\widehat{\vx}-\vx_0\|_2, \|\widehat{\vx}+\vx_0\|_2\}\gtrsim\|\ve\|_2/\sqrt{m}$ when $|\sum_{i=1}^m\eta_i|\geq \delta_0$ and $\|\ve\|_2/\sqrt{m}\leq \delta_1$ for some positive constants $\delta_0$ and $\delta_1$.}}
 A similar result is  obtained for  (\ref{eqn: sparse_model}) provided $m \gtrsim k \log (ed/k)$, where $k:=\|\vx_0\|_0$  \cite{sparse2}.
  However, the methods employed in \cite{Real_nonlinear_LS, sparse2} heavily depend on $\vx_0$ being real. More specifically, the modulo phase in the real case is only  $\pm1$-sign, but, in the complex case, it becomes the complex unit circle, which is an uncountable set.
   This difference makes it highly nontrivial to extend the results from the real case to the complex case. This paper overcomes this challenge by introducing new technical tools.
   Additionally, we also evaluate the performance of the amplitude-based models for complex signals.  Particularly, for any solution $\widehat{\vx}\in \mathbb{C}^{d}$ to (\ref{eqn: nonsparse_model}), it can be shown that
   \begin{equation}\label{eqn: error_complex_main}
 \min_{\theta\in[0,2\pi)}\|\widehat{\vx}-\exp(\mathrm{i}\theta)\vx_0\|_2\lesssim \frac{\|\ve\|_2}{\sqrt{m}},
 \end{equation}
 provided that $m\gtrsim d$.
  { The solutions to  the model (\ref{eqn: nonsparse_model}) are not unique  (see \cite{HX21}) even if we consider solutions that differ by a unimodular constant factor as the same solution. However, for any given solution $\widehat{\vx}$, the inequality (\ref{eqn: error_complex_main}) holds true.
 } Last but not least, the results of this paper require that the $\va_j, j=1,\ldots,m$,  are  {i.i.d.} sub-Gaussian measurements, when $m\gtrsim d$.  {In fact, random phase retrieval based on sub-Gaussian distribution arises in practical applications, such as imaging in complex media \cite{M}.   In these scenarios, the random linear operator is derived from multiple light scattering or random input patterns synthesized by a programmable device.
 More specifically, the authors in \cite{M} note that the measurement vectors in these scenarios are naturally i.i.d. sub-Gaussian.
 Sub-Gaussian measurements encompass Gaussian measurements as a special case. However, the proof techniques used differ significantly due to the loss of rotation invariance property in sub-Gaussian distributions. Furthermore, certain types of sub-Gaussian measurements, such as the Bernoulli distribution, are incapable of successfully recovering the signal in phase retrieval problems. This necessitates the introduction of additional assumptions on the sub-Gaussian distribution. More detailed information can be found in Section 1.4.
 }


\subsection{Related work}
The investigations of models for phase retrieval are divided into two cases: one for reals signals and the other for complex ones.  We next introduce the known results  for each of these cases.  Additionally, we will also introduce relevant  algorithms for solving the amplitude-based model (\ref{eqn: nonsparse_model}).
\subsubsection{Real signals}
As mentioned before,   if $\va_j,$ $j=1,\ldots,m,$ are  {i.i.d.} real Gaussian random vectors, the solution $\widehat{\vx}$ to (\ref{eqn: nonsparse_model})  satisfies
\begin{equation}\label{eq:upb}
\min\{\|\widehat{\vx}-\vx_0\|_2, \|\widehat{\vx}+\vx_0\|_2\}\lesssim\|\ve\|_2/\sqrt{m}
\end{equation}
provided  that $m \gtrsim d$ \cite{Real_nonlinear_LS}. It is also shown that the bound $\|\ve\|_2/\sqrt{m}$ is tight up to a constant \cite{Real_nonlinear_LS}. There exist algorithms that can almost achieve the performance upper bound. 
 The performance of the alternating projection algorithm to deal with the model (\ref{eqn: nonsparse_model}) is investigated for the noiseless case, i.e., $\ve=\boldsymbol{0}$ \cite{AM15, Fienup_converge}.
For the noisy case, some generalized gradient-type algorithms, such as Amplitude Flow (AF) and Truncated Amplitude Flow (TAF), are discussed in \cite{AF, TAF, PAF}.  Particularly, when $\va_j,$  $j=1,\ldots,m,$ are independent real Gaussian random vectors with $m\gtrsim d$,  the sequence $\{\vz^{(k)}\}_{k\geq 1}$, which is generated by AF or TAF, obeys
\begin{equation}\label{eq:AF}
\min\{\|\vz^{(k)}+\vx_0\|_2,\|\vz^{(k)}-\vx_0\|_2\}\lesssim \frac{\|\ve\|_2}{\sqrt{m}}+\delta^k\|\vx_0\|_2
\end{equation}
 for some $\delta\in (0,1)$ {(see  \cite{AF, TAF})}.    A simple observation is that
  the right-hand side of (\ref{eq:AF}) tends to reach the performance upper bound ${\|\ve\|_2}/{\sqrt{m}}$ when $k$ tends to infinity. However, the convergence result is obtained under the assumption of  $\|\boldsymbol{\ve}\|_\infty\lesssim\|\vx_0\|_2$.


In  \cite{TWF},  Chen and Cand\`es  designed the Truncated Wirtinger Flow (TWF)  algorithm for solving
the Poisson data model:
\[
\argmin{\vx\in \mathbb{R}^d}\,\, -\sum_{j=1}^m\left(y_j\log(|\langle \va_j,\vx\rangle|^2)-|\langle \va_j, \vx\rangle|^2\right).
\]
They showed that if $\va_j$, $j=1,\ldots,m,$ are  {i.i.d.} real Gaussian random vectors with $m\gtrsim d$,
 the sequence $\{\vz^{(k)}\}_{k\geq 1}$ generated by TWF satisfies
\[
\min\{\|\vz^{(k)}+\vx_0\|_2,\|\vz^{(k)}-\vx_0\|_2\}\,\,\lesssim\,\, \frac{\|\ve\|_2}{\sqrt{m}\|\vx_0\|_2}+\delta^k\|\vx_0\|_2
\]
 for some $\delta\in (0,1)$, provided that $\|\ve\|_\infty\lesssim \|\vx_0\|_2^2$.

\subsubsection{Complex signals}
In this section, we introduce the known results for the recovery performance of phase retrieval for complex signals.
 In noisy case, most of  the historical works are based on the intensity-based observations,  i.e.,
\begin{equation}\label{eqn: intensity_observation}
y_j=\abs{\langle \va_j,\vx_0\rangle}^2+\epsilon_j,\qquad j=1,\ldots,m.
\end{equation}



To  estimate $\vx_0\in \mathbb{C}^{d}$ from (\ref{eqn: intensity_observation}), both convex and nonconvex methods are developed. The convex method applies the ``lifting" technique, that is, lifting the signal vector $\vx$ into rank-one matrix $\vX=\vx\vx^*$ and then the quadratic measurements are linearized. The convex relaxation was introduced in \cite{Phaselift1, Phaselift2}. In \cite{Phaselift1}, Cand\`es and Li considered the following model:
\begin{equation}\label{eqn: convex_relaxation}
\argmin{\vX\in \mathbb{C}^{d\times d}}\ \sum_{j=1}^m\left|\va_j^*\vX\va_j-y_j\right|\qquad\text{s.t.}\quad  \vX\succeq0.
\end{equation}
They proved that the solution $\widetilde{\vX}$ to (\ref{eqn: convex_relaxation}) satisfies
\begin{equation*}
\|\widetilde{\vX}-\vx_0\vx_0^*\|_F\,\,\lesssim\,\, \frac{\|\boldsymbol{\epsilon}\|_1}{m}
\end{equation*}
with $\boldsymbol{\epsilon}=(\epsilon_1,\ldots,\epsilon_m)^{\top}$, provided  that $\va_j,$ $j=1,\ldots,m,$ are {i.i.d.} complex Gaussian random vectors with $m\gtrsim d$.{ {The measurements in (\ref{eqn: convex_relaxation}) can also be  extended into a more generalized setting, where $\va_j,$ $j=1,\ldots,m,$ are i.i.d. sub-Gaussian random vectors \cite{Krahmer_subgaussian}.}}
Assume that  $\widetilde{\vx}$ is the largest eigenvector corresponding to the largest eigenvalue of $\widetilde{\vX}$.
According to the error bound presented in  \cite{Phaselift1,Krahmer_subgaussian}, we have
\[
\min_{\theta\in[0,2\pi)}\|\widetilde{\vx}-\exp(\mathrm{i}\theta)\vx_0\|_2\lesssim \min\Big\{\|\vx_0\|_2,\frac{\|\boldsymbol{\epsilon}\|_1}{\|\vx_0\|_2\cdot m}\Big\}.
\]

Since the parameter space in convex relaxation is usually much larger than that of the original problem, a renewed interest is focused directly on optimizing nonconvex formulation:
\begin{equation}\label{eqn: smooth}
\argmin{\vx\in \mathbb{C}^d}\ \sum_{j=1}^m\left(|\langle \va_j,\vx\rangle|^2-y_j\right)^2.
\end{equation}
For the noiseless case,  Cand\`es,   Li, and  Soltanolkotabi developed the Wirtinger Flow method (WF) for solving (\ref{eqn: smooth}) \cite{WF}.
For the noisy case,  Huang and Xu studied the performance of  (\ref{eqn: smooth}) showing that the solution $\widehat{\vx}\in \mathbb{C}^{d}$ to (\ref{eqn: smooth}) satisfies
\begin{equation}\label{eqn: nonconvex_forth}
 \min_{\theta\in[0,2\pi)}\|\widehat{\vx}-\exp(\mathrm{i}\theta)\vx_0\|_2\lesssim\min\left\{\frac{\sqrt{\|\boldsymbol{\epsilon}\|_2}}{m^{1/4}},\frac{\|\boldsymbol{\epsilon}\|_2}{\|\vx_0\|_2\sqrt{m}}\right\},
\end{equation}
provided $m\gtrsim d$ \cite{forthmodel_PR}.  {Furthermore, Sun, Qu, and  Wright in \cite{SQW} established some landscape result in noiseless case, when $m\gtrsim d\log^3 d$. That is, the only local  that are also global minimizers of (\ref{eqn: smooth})  are the target $\exp(\mathrm{i}\theta)\vx_0$ for $\theta\in[0,2\pi)$, and the saddle points and local maximizer have negative curvature in some certain direction. Therefore, any algorithm that  successfully escape from this kind of saddle point  can recover the target signal $\vx_0$ up to a unimodular constant.}
\subsubsection{Sparsity-based models}
 The recovery of sparse signals $\vx_0\in \F^d$ from phaseless measurements is an active topic recently, which is referred to as the  {\em compressive phase retrieval} problem \cite{sparse1,sparse2,sparse3}. In this section, we assume that $\|\vx_0\|_0\leq k$.
 The performance of the $\ell_1$-model for recovering real sparse signals $\vx_0$ is studied  in \cite{sparse2}:
 \begin{equation}\label{eq:L1phase}
 \argmin{\vx\in \R^d}\|\vx\|_1\quad {\rm s.t.}\quad (\sum_{j=1}^m(\abs{\innerp{\va_j,\vx}}-b_j)^2)^{1/2}\leq \epsilon.
 \end{equation}
where  $\{\va_j\}_{j=1}^m$ are  {i.i.d.} real Gaussian random vectors and $b_j=\abs{\innerp{\va_j,\vx_0}}+\eta_j,  j=1,\ldots,m$, with $\|\ve\|_2\leq \epsilon$. The solution $\hat{\vx}$ to (\ref{eq:L1phase}) satisfies
\[
\min \{\|\hat{\vx}-\vx_0\|, \|\hat{\vx}+\vx_0\|\}\lesssim \epsilon/\sqrt{m}
\]
provided that $m\gtrsim k\log (ed/k)$. The method employed in  \cite{sparse2} heavily depends on the real space property like ${\rm sign}(\innerp{\va_j,\vx})\in \{-1,1\}$.

{{The intensity-based model  for complex sparse signals is discussed in \cite{XiaXu}, that is,
 \begin{equation}\label{eq:L1_forth}
 \min_{\vx\in \mathbb{C}^d}\|\vx\|_1\quad {\rm s.t.}\quad (\sum_{j=1}^m(\abs{\innerp{\va_j,\vx}}^2-y_j)^2)^{1/2}\leq \epsilon,
 \end{equation}
 where $y_j=\abs{\langle \va_j,\vx_0\rangle}^2+\epsilon_j$, $j=1,\ldots,m$, with $\|\boldsymbol{\epsilon}\|_2\leq \epsilon$. The solution $\hat{\vx}$ to (\ref{eq:L1_forth}) becomes
 \[
\min_{\theta\in[0,2\pi)}\|\hat{\vx}-\exp(\mathrm{i}\theta)\vx_0\|_2\lesssim \min\Big\{\|\vx_0\|_2+\|\vx_0\|_1,\frac{\epsilon}{\|\vx_0\|_2\cdot \sqrt{m}}\Big\}.
\]}}
\subsubsection{Algorithms}
Although the amplitude-based model (\ref{eqn: nonsparse_model}) is nonconvex,  several efficient algorithms have been developed to solve it.  A popular algorithm for solving (\ref{eqn: nonsparse_model})  is the alternating projection algorithms ( see Fienup \cite{Fienup1}, Gerchberg and Saxton \cite{Fienup2}).
 {The alternating projection algorithm can be equivalently formulated as solving the following model:
\begin{equation}\label{eqn: GS}
\min_{\boldsymbol{x}\in \mathbb{C}^d,\ \boldsymbol{c}\in \mathbb{C}^m} \|\boldsymbol{A}\boldsymbol{x}-\boldsymbol{c}\odot \boldsymbol{b}\|_2,
\end{equation}
where  $\boldsymbol{c}=(c_1,\ldots,c_m)\in \mathbb{C}^m$  with $|c_j|=1$, $j=1,\ldots, m$ and $\odot$ is element-wise multiplication operator. Through a simple calculation, it can be observed that if the entries of $\boldsymbol{b}$ are nonnegative, $\widehat{\boldsymbol{x}}$ is a solution to the amplitude-based model if and only if $(\widehat{\boldsymbol{x}}, \text{sign}(\boldsymbol{A}\widehat{\boldsymbol{x}}))$ is a solution to (\ref{eqn: GS}).
The alternating projection algorithm employs the well-known alternating minimization technique, where $\boldsymbol{x}$ and $\boldsymbol{c}$ are updated alternatively to minimize (\ref{eqn: GS}) \cite{Fienup1, Fienup2}.
In cases where measurements are noiseless, Waldspurger \cite{Fienup_converge} has shown convergence results for complex Gaussian measurements when $m\gtrsim d$. Although there is a lack of theoretical guarantees for generalized Wirtinger flow methods like TAF \cite{TAF} and Reshaped WF \cite{RWF} in the complex scenario, numerical results suggest that these algorithms demonstrate commendable performance in complex situations.
}
\subsection{Notations}
 Denote $\mathbb{H}^{d\times d}$ as the $d\times d$ Hermitian matrices set. For a matrix $\vX\in
\mathbb{C}^{d_1\times d_2}$, we use
$\vX_{S,T}$ to denote the submatrix of $\vX$ with the rows indexed in $S$ and columns
indexed in $T$.  We use $\vX_{l,:}$  and $\vX_{:,j}$  to denote the $l$-th row and the
$j$-th column of $\vX$, respectively. Take $\|\vX\|_{0,2}$ as  the number of
non-zero columns in $\vX$.
Set
$\|\vX\|_1:=\sum_{j,l}{\sqrt{\text{R}(X_{j,l})^2+\text{I}(X_{j,l})^2}}$,
$\|\vX\|_F:=\sqrt{\sum_{j,l}(\text{R}(X_{j,l})^2+\text{I}(X_{j,l})^2)}$ and
$\|\vX\|_{1,2}:=\sum_{j}\|\vX_{:,j}\|_2$, where $\text{R}(X_{j,l})$ and $\text{I}(X_{j,l})$ are the real and imaginary parts of $X_{j,l}$, respectively. For any $\vX, \boldsymbol{Y}\in \mathbb{C}^{d_1\times d_2}$, set $\langle \vX, \boldsymbol{Y}\rangle
:=\text{Tr}(\vX^*\boldsymbol{Y})$.

    We use $ A \gtrsim B$ to denote $A\ge c B$,
where $c$ is some positive absolute constant. The notation $\lesssim$ can be defined
similarly.
Without specific notation, we use $C, c$ and their superscript
(subscript)   forms to denote universal constants and their values may vary with different contexts.
\subsection{Our contribution}
As mentioned before,  the reconstruction errors of (\ref{eqn: nonsparse_model}) and (\ref{eqn: sparse_model}) were only studied for $\F=\R$ and real Gaussian measurements. {{In this paper, we consider the case of $\F=\mathbb{C}$ and {\em sub-Gaussian measurements}. The complex random measurement vectors $\{\va_j\}_{j=1}^m$ are taken as independent copies of a random vector $\va=(a_1,\ldots, a_d)\in \mathbb{C}^{d}$, whose elements $a_j, j\in\{1,\ldots,d\}$ are independently chosen from some sub-Gaussian distribution defined as below:
\begin{definition}
A random variable $\xi$ is called a sub-Gaussian random variable, if there exist constants $\alpha, \kappa>0$ such that $\mathbb{P}(|\xi|\geq t)\leq \alpha \exp(-\kappa t^2)$ for all $t>0$. The sub-Gaussian norm of $\xi$ is denoted as
\begin{equation}\label{eqn: C_phi}
\|\xi\|_{\psi_2}=\sup_{p\geq 1} p^{-1/2} (\mathbb{E}|\xi|^p)^{1/p}.
\end{equation}
\end{definition}
For simplicity, throughout this paper, we can assume that $\mathbb{E}a_j=0$, $\mathbb{E}|a_j|^2=1$ and $\mathbb{E}|a_j|^4\neq 1$. The assumption that the fourth moment $\mathbb{E}|a_j|^4\neq 1$ is made primarily because unique recovery may not be possible without this assumption. For example, take $a_j=\frac{\sqrt{2}}{2}\xi_{j,1}+\frac{\sqrt{2}}{2}\xi_{j,2}\mathrm{i}$ with $\xi_{j,1}$ and $\xi_{j,2}$ to be i.i.d. Rademacher random variables that take values in $\{ +1, -1\}$ with equal probability.  Then the measurements will always yield the same observations $\boldsymbol{y}$ for the following two signals:
\[
\vx_1=(1,0,\ldots,0),\quad\text{and}\quad \vx_2=(0,1,\ldots,0).
\]

If  $\mathbb{E}\abs{a_j}^2=1$ and $\mathbb{E}\abs{a_j}^4\neq 1$,  Jensen's inequality implies that $\mathbb{E}\abs{a_j}^4>1$.}}  {For simplicity, we assume that $a_{j}=a_{j,1}+a_{j,2}\mathrm{i}$, where $a_{j,1}$ and $a_{j,2}$ are independent copies of some real sub-Gaussian random variables. It leads to  $\mathbb{E}a_{j}^2=0$. We put the assumption $\mathbb{E}a_{j}^2=0$ here for the convenience of the proofs in the main results.   }
Therefore, in the following statements, we consider the measurement elements drawn from some sub-Gaussian random variable $\xi=\xi_1+\xi_2\mathrm{i}$ with $\|\xi\|_{\psi_2}\leq C_{\psi}$, $\mathbb{E}\xi=0$, $\mathbb{E}\xi^2=0$, $\mathbb{E}|\xi|^2=1$, and $\mathbb{E}|\xi|^4=\gamma>1$, where $\xi_1$ and $\xi_2$ are independent copies of some real sub-Gaussian random variable.

The following theorem establishes the performance of the amplitude-based model  (\ref{eqn: nonsparse_model})  for $\F=\mathbb{C}$.
We want to mention that the proof strategy of Theorem \ref{th:upbound} is quite different from that in \cite{Real_nonlinear_LS}.
\begin{theorem}\label{th:upbound}
Assume that
 $\va_j\in \mathbb{C}^d,  \ j=1,\ldots,m,$ are  i.i.d.
complex sub-Gaussian random vectors, whose elements are independently drawn from some sub-Gaussian random variable $\xi$ with $\|\xi\|_{\psi_2}\leq C_{\psi}$, $\mathbb{E}\xi=0$, $\mathbb{E}\xi^2=0$, $\mathbb{E}|\xi|^2=1$, and $\mathbb{E}|\xi|^4=\gamma>1$.  Assume that $m\gtrsim d$. Then the following holds with probability at least $1-2 \exp(-c_0m/4)$:
for all $\vx_0\in \mathbb{C}^d$ and $\ve=(\eta_1,\ldots,\eta_m)\in \R^m$, the solution  $\widehat{\vx}$ to (\ref{eqn: nonsparse_model}) with
$\mathbb{F}=\mathbb{C}$ and $b_j=\abs{\innerp{\va_j,\vx_0}}+\eta_j$, $j=1,\ldots,m$, satisfies
 \[
 \min_{\theta\in[0,2\pi)}\|\widehat{\vx}-\exp(\mathrm{i}\theta)\vx_0\|_2\lesssim \frac{\|\ve\|_2}{\sqrt{m}}.
 \]
Here, $c_0$ is a positive constant depending on $C_{\psi}$ and $\gamma$.
\end{theorem}
\begin{remark}{
In \cite{albert}, Fannjiang and  Strohmer discussed another type of noise, i.e., ${b}_j=\left|\langle\va_j, \vx_0 \rangle+\widetilde{\eta}_j\right|$, $j=1,\ldots,m$, where $\widetilde{\ve}:=(\widetilde{\eta}_1,\ldots,\widetilde{\eta}_m)^{\top}$ is a noise term. Since ${b}_j$ can be rewritten as ${b}_j=\left|\langle\va_j, \vx_0 \rangle\right|+\eta_j$ with $\eta_j:=\left|\langle\va_j, \vx_0 \rangle+\widetilde{\eta}_j\right|-\left|\langle\va_j, \vx_0 \rangle\right|$, a simple calculation shows that $\left|{\eta}_j\right|\leq \left|\widetilde{\eta}_j\right|$, and hence we can use Theorem \ref{th:upbound} to show that the solution  $\widehat{\vx}$ to (\ref{eqn: nonsparse_model}) with
$\mathbb{F}=\mathbb{C}$ and $b_j=\abs{\innerp{\va_j,\vx_0}+\widetilde{\eta}_j}$, $j=1,\ldots,m$, satisfies
 \[
 \min_{\theta\in[0,2\pi)}\|\widehat{\vx}-\exp(\mathrm{i}\theta)\vx_0\|_2\lesssim \frac{\|\widetilde{\ve}\|_2}{\sqrt{m}}.
 \]}
\end{remark}

{{The result in Theorem \ref{th:upbound} extends the error bound (\ref{eq:upb}) for the real signals from  \cite{Real_nonlinear_LS} to complex ones by employing different tools. In \cite{Real_nonlinear_LS}, since  ${\rm sign}(\innerp{\va_j,\vx})$ only contains two values $\{+1, -1\}$ in the real case, the error bound is analyzed by dividing  ${\rm sign}(\innerp{\va_j,\vx})$ into two separable sets. As the phase in the complex case lies in the complex unit circle, the separation technique in the real case can not be extended to the complex one. To resolve the phase ambiguity, we employ the lifting technique and examine the uniform concentration inequality on low-rank and sparse matrices (refer to Theorem \ref{thm: SRIP}), which plays a crucial role in the proof.  {
The tools employed in the proof of Theorem \ref{th:upbound} can be easily extended to the real case.
 However, as the results for the real case have been previously established in \cite{Real_nonlinear_LS}, we will not delve into the detailed expansion. }}}


The following theorem presents a lower bound for the estimation error for any fixed $\vx_0$, which implies that the reconstruction error presented in Theorem \ref{th:upbound} is tight.
\begin{theorem}\label{thm: sharpness}
Assume that  $\ve=(\eta_1,\ldots,\eta_m)\in \R^m$ satisfies
\begin{equation}\label{eqn: eta_condition}
\left|\sum_{j=1}^m\eta_j\right|\geq C \sqrt{m}\|\boldsymbol{\ve}\|_2,
\end{equation}
for some absolute constant $C\in (0,1)$.  Under the assumption of Theorem \ref{th:upbound} about $\va_j\in \mathbb{C}^d, j=1,\ldots,m$, for any fixed $\vx_0\in \mathbb{C}^d$ satisfying $\|\vx_0\|_2\gtrsim \frac{\|\ve\|_2}{\sqrt{m}} $,
the following holds with probability at least $1-6\exp(-cm)$:
the solution  $\widehat{\vx}$ to (\ref{eqn: nonsparse_model}) with $\mathbb{F}=\mathbb{C}$ and
$b_j:=\abs{\innerp{\va_j,\vx_0}}+\eta_j,$ $j=1,\ldots,m,$ satisfies
\begin{equation}\label{eqn: lower_error}
 \min_{\theta\in[0,2\pi)}\|\widehat{\vx}-\exp(\mathrm{i}\theta)\vx_0\|_2\gtrsim \frac{\|\ve\|_2}{\sqrt{m}},
\end{equation}
provided  $m\gtrsim d$.  Here, $c$ is a positive constant depending on $C_{\psi}$ and $\gamma$.
\end{theorem}
\begin{remark}
 Theorem  \ref{thm: sharpness} requires that $\ve$ satisfies (\ref{eqn: eta_condition}).
 In fact, a lot of $\ve\in \R^m$ satisfy (\ref{eqn: eta_condition}).  {{For instance, taking $\eta_j\in[a,b]$ with $0<a<b$, we have
 \[
 \left|\sum_{j=1}^m\eta_j\right|\geq \frac{a}{b}\cdot\sqrt{m}\|\ve\|_2.
 \]
 Moreover, $\eta_j$ can be taken as random noise. Consider $\eta_j\overset{i.i.d.}\sim\mathcal{N}(\mu,\sigma^2)$. For convenience, assume that  {$\mu>0$}  ($\mu< 0$ can be discussed similarly).}} Then we have $\sum_{j=1}^m \eta_j\sim \mathcal{N}(\mu m, \sigma^2 m)$ which implies that
\begin{equation}\label{eq:t1t2}
\mathbb{P}\Big(\Big|\sum_{j=1}^m\eta_j-\mu m\Big|\geq \sigma t_1\Big)\leq 2\exp\left(-\frac{ct_1^2}{m}\right)\ \text{and}\  \
\mathbb{P}\left(\left|\sum_{j=1}^m(\eta_j-\mu )^2-\sigma^2 m\right|\geq \sigma ^2 t_2\right)\leq 2\exp\left(-\frac{ct_2^2}{m}\right).
\end{equation}
Taking $t_1=c_1 m$ and $t_2=c_2 m$ in (\ref{eq:t1t2}), we obtain that the following holds with probability at least $1-2\exp\left(-c_1cm\right)-2\exp\left(-{c_2cm}\right)$:
\begin{equation}\label{eqn: eta_square}
(\mu-c_1\sigma)m \leq \sum_{j=1}^m\eta_j\leq (\mu+c_1\sigma)m\qquad\text{and}\qquad \sigma^2(1-c_2)m\leq \sum_{j=1}^m (\eta_j-\mu)^2\leq \sigma^2(1+c_2)m,
\end{equation}
which implies
\[
\sum_{j=1}^m \eta_j^2\leq 2\mu\sum_{j=1}^m\eta_j+(-\mu^2+\sigma^2(1+c_2))m\leq 2\mu(\mu+c_1\sigma)m+(-\mu^2+\sigma^2(1+c_2))m.
\]
To guarantee  $\ve$ satisfy $\left|\sum_{j=1}^m\eta_j\right|\geq C \sqrt{m}\|\boldsymbol{\ve}\|_2$, it is enough to require that
\[
\mu-c_1\sigma \geq C\sqrt{2\mu(\mu+c_1\sigma)-\mu^2+\sigma^2(1+c_2)},
\]
when $c_1<\mu/\sigma$,
which is equivalent to
\begin{equation}\label{eq:Ce}
C\leq \frac{\mu-c_1\sigma}{\sqrt{2\mu(\mu+c_1\sigma)-\mu^2+\sigma^2(1+c_2)}}.
\end{equation}
We can choose $\sigma>0$ is small enough, say $\sigma=\sigma_0$, so that (\ref{eq:Ce}) holds. Hence, if $\eta_j\overset{i.i.d.}\sim\mathcal{N}(\mu,\sigma_0^2)$ with $\mu> 0$, then  (\ref{eqn: eta_condition}) holds with high probability.
\end{remark}

 {
In Theorem \ref{thm: sharpness}, we necessitate a large value for $\abs{\sum_{j=1}^m\eta_j}$. Nevertheless, we put forth Theorem \ref{thm: better_estimation} to illustrate the potential for a superior estimation when $\abs{\sum_{j=1}^m\eta_j}$ is of a smaller magnitude.
\begin{theorem}\label{thm: better_estimation}
 Assume that $m\gtrsim d$.
 Assume that $\eta=(\eta_1,\ldots,\eta_m)\in \R^m$ satisfies
\begin{equation*}\label{eq:etacon}
 \left|\sum_{j=1}^m\eta_j\right|\lesssim  \sqrt{m\log m}\qquad\text{and}\qquad \sum_{j=1}^m \eta_j^2\lesssim m.
\end{equation*}
 Under the assumption of Theorem \ref{th:upbound} about $\va_j\in \mathbb{C}^d, j=1,\ldots,m$, the following holds with probability at least $1-2 \exp(-c_0m/4)$:
for all $\vx_0\in \mathbb{C}^d$, the solution  $\widehat{\vx}$ to (\ref{eqn: nonsparse_model}) with
$\mathbb{F}=\mathbb{C}$ and $b_j=\abs{\innerp{\va_j,\vx_0}}+\eta_j$, $j=1,\ldots,m$, satisfies
 \begin{equation}\label{eqn: better_estimation_new}
 \min_{\theta\in[0,2\pi)}\|\widehat{\vx}-\exp(\mathrm{i}\theta)\vx_0\|_2\leq C_0 \max\left\{\|\vx_0\|_2,\sqrt{\frac{d\log m}{m}}\right\}\cdot \sqrt{\frac{d\log m}{m}}.
 \end{equation}
Here, $c_0$  and $C_0$ are  positive constants depending on $C_{\psi}$ and $\gamma$.
\end{theorem}
\begin{remark}
Assume  $\ve=(\eta_1,\ldots,\eta_m)\in \R^m$ with $\eta_j$ independently drawn from $\mathcal{N}(0,\sigma^2)$.
 Taking $\mu=0$, $t_1=\sqrt{m\log m/c}$ and $t_2=\sqrt{m/c}$ in (\ref{eq:t1t2}), we obtain that
\begin{equation}\label{eqn: upper_eta_m0}
 \left|\sum_{j=1}^m\eta_j\right|\lesssim \sigma \sqrt{m\log m}\qquad\text{and}\qquad \sum_{j=1}^m \eta_j^2\lesssim \sigma^2m,
\end{equation}
hold with probability at least $1-2/m-2\exp(-m)$.
Combining Theorem \ref{thm: better_estimation} and (\ref{eqn: upper_eta_m0}), the following holds
\begin{equation}\label{eqn: gauetabetter_estimation_new}
 \min_{\theta\in[0,2\pi)}\|\widehat{\vx}-\exp(\mathrm{i}\theta)\vx_0\|_2\leq C_0 \max\left\{\|\vx_0\|_2,\sqrt{\frac{d\log m}{m}}\sigma\right\}\cdot \sqrt{\frac{d\log m}{m}}\sigma
 \end{equation}
 with probability at least
 \begin{equation}\label{eqn: better_probability}
 1-2\exp(-c_0m/4)-2/m.
 \end{equation}
As $m$ approaches infinity, the error bound in  (\ref{eqn: gauetabetter_estimation_new}) converges to zero. In the context of zero-mean Gaussian noise, the error estimation in (\ref{eqn: gauetabetter_estimation_new}) significantly outperforms the result in Theorem \ref{th:upbound}, which states that
$
 \min_{\theta\in[0,2\pi)}\|\widehat{\vx}-\exp(\mathrm{i}\theta)\vx_0\|_2\lesssim \frac{\|\boldsymbol{\eta}\|_2}{\sqrt{m}} \thickapprox \sigma.
$
However, the probability bound (\ref{eqn: better_probability})  is significantly smaller than the value of $1-2 \exp(-c_0m/4)$ as stated in Theorem \ref{th:upbound}.
The question of whether the bound in Theorem \ref{thm: better_estimation} is tight or not is a topic that will be explored in future research.
 It is worth noting that the error estimation associated with the intensity-based model, which takes into account zero-mean Gaussian noise, has been  studied in \cite{T. Cai, subgaussian}.
\end{remark}
}
We next turn to the complex sparse phase retrieval problem. The vector $\vx_0\in \mathbb{C}^d$ is called {\em $k$-sparse}, if there are at most $k$ nonzero elements in $\vx_0$. The performance of constrained $\ell_1$ model (\ref{eqn: sparse_model}) is presented  in
the following theorem.
\begin{theorem}\label{thm: sparse_result}
Assume that     $m\gtrsim k\log(ed/k)$.
Under the assumption of Theorem \ref{th:upbound} about $\va_j\in \mathbb{C}^d, j=1,\ldots,m$,
 the following holds with probability at least $1-2 \exp(-c_0m/4)$: for all the $k$-sparse vector $\vx_0\in \mathbb{C}^d$ and $\ve=(\eta_1,\ldots,\eta_m)\in \R^m$ with $\|\boldsymbol{\eta}\|_2\leq \epsilon$, the solution  $\widehat{\vx}$ to (\ref{eqn: sparse_model}) with
$\mathbb{F}=\mathbb{C}$ and $b_j:=\abs{\innerp{\va_j,\vx_0}}+\eta_j,$ $j=1,\ldots,m$ satisfies
 \begin{equation}\label{eqn: sparse_error}
 \min_{\theta\in[0,2\pi)}\|\widehat{\vx}-\exp(\mathrm{i}\theta)\vx_0\|_2\,\,\,\lesssim\,\,\, \frac{\epsilon}{\sqrt{m}}.
 \end{equation}
Here, $c_0$ is the same constant as that in Theorem \ref{th:upbound}.
 \end{theorem}
{ {
 \begin{remark}
 {In \cite{HSX}, the authors \cite{HSX} provide a comprehensive discussion on the amplitude-based model in the context of real signals and the intensity-based model in the context of complex signals, specifically focusing on the scenario of sparse signals. The model in \cite{HSX} has some non-zero additional bias. More explicitly,  they presented the performance of the following $\ell_1$ minimization program:
\begin{equation}\label{eqn: affine}
\arg\min_{\vx\in \mathbb{R}^d}\ \|\vx\|_1\qquad \text{s.t.}\ \||\boldsymbol{A}\vx+\vw|-\vb\|_2\leq \epsilon.
\end{equation}
Here $\vw$ is a non-zero vector satisfying $\alpha<\|\vw\|_2< \beta$ with $\alpha,\beta>0$, and $\vb=|\vA\vx_0+\vw|+\boldsymbol{\eta}$ with $\|\boldsymbol{\eta}\|_2\leq \epsilon$. In the real case, the solution $\widehat{\vx}$ to (\ref{eqn: affine}) obeys
\[
\|\widehat{\vx}-\vx_0\|_2\lesssim \frac{\epsilon}{\sqrt{m}}.
\]
Compared with the result in (\ref{eq:upb}), the appearance of the non-zero bias vector $\vw$ helps to resolve the phase ambiguity in sparse phase retrieval problems. However, when the bias vector is set to zero, it fails to meet the conditions required in \cite{HSX}. As a result, the findings in \cite{HSX} cannot be directly applied to the phase retrieval problem discussed in our paper.}
\end{remark}
\begin{remark}
An intriguing question that emerges is the feasibility of replacing the term ${\epsilon}/{\sqrt{m}}$ on the right-hand side of equation (\ref{eqn: sparse_error}) with ${\epsilon}/{m^\delta}$, where $\delta$ is a fixed constant greater than $1/2$. Indeed, there exist $\vx_0\in \mathbb{C}^d$, $\boldsymbol{\eta}\in \mathbb{R}^n$ and $\epsilon>0$ such that $\min_{\theta\in[0,2\pi)}\|\widehat{\vx}-\exp(\mathrm{i}\theta)\vx_0\|_2\gtrsim\frac{\epsilon}{\sqrt{m}}$ holds with probability at least $1/2$. This demonstrates that the error bound given in (\ref{eqn: sparse_error}) is sharp up to a constant.  For instance, take $\vx_0=(1,0,\ldots, 0)$, $\boldsymbol{\eta}=(1,\ldots, 1)$ and $\epsilon=\sqrt{8m}$.
 Combining Markov inequality and
 \[
 \mathbb{E}\sum_{j=1}^m b_j^2=\mathbb{E}(|a_{j,1}|+1)^2=\sum_{j=1}^m\mathbb{E}(|a_{j,1}|^2+2|a_{j,1}|+1)\leq 4m,
 \]
 we obtain that
 \begin{equation}\label{eq:l10}
 \mathbb{P}(\sum_{j=1}^m\left(|\langle \va_j,\boldsymbol{0}\rangle|-b_j\right)^2> \epsilon)=\mathbb{P}\big(\sum_{j=1}^m b_j^2> \epsilon^2\big)\leq \frac{ \mathbb{E}\sum_{j=1}^m b_j^2}{\epsilon^2}\leq \frac{4m}{8m}=\frac{1}{2},
 \end{equation}
  which implies the solution  $\widehat{\vx}$ to (\ref{eqn: sparse_model}) is $\widehat{\vx}=\boldsymbol{0}$ with probability at least $1/2$. Thus
\[
\min_{\theta\in[0,2\pi)}\|\widehat{\vx}-\exp(\mathrm{i}\theta)\vx_0\|_2\,\,\,\gtrsim\,\,\, \frac{\epsilon}{\sqrt{m}}
\]
holds with probability at least $1/2$.
\end{remark}

\section{Proof of Theorem \ref{th:upbound}}
Assume that $\va_j\in \mathbb{C}^d,$  $j=1,\ldots,m$.
Let $\cA: \mathbb{H}^{d\times d}\rightarrow \mathbb{C}^m$ be a linear operator corresponding to $\{\va_j\}_{j=1}^m$, which is defined as
\begin{equation}\label{eqn: A()}
\cA(\vX)\,\,=\,\,(\va_1^*\vX\va_1,\ldots, \va_m^*\vX\va_m),
\end{equation}
where $\vX\in \mathbb{H}^{d\times d}$.  Additionally, for $I\subset \{1,\ldots,m\}$, let
$\cA_I: \mathbb{H}^{d\times d}\rightarrow \mathbb{C}^{\# I}$ be a linear operator corresponding to $\{\va_j\}_{j\in I}$ which is defined as
\[
\cA_I(\vX)\,\,=\,\,(\va_j^*\vX\va_j)_{j\in I}.
\]
We introduce the following results, which are helpful for the proofs.
\begin{lemma}\cite[Lemma 3.2]{XiaXu}\label{Lemma 3.2}
If $\vx,\vy\in \mathbb{C}^d$, and $\langle \vx, \vy\rangle\in \mathbb{R}$ with $\langle \vx, \vy\rangle\geq 0$, then
\[
\|\vx\vx^{*}-\vy\vy^{*}\|_F^2\geq \frac{1}{2}\cdot \|\vx\|_2^2 \cdot \|\vx-\vy\|_2^2.
\]
\end{lemma}
 {The following lemma is based on   the H$\ddot{\text{o}}$lder's inequality and
Proposition 2.6.1 in \cite{Vershynin}.
  The H$\ddot{\text{o}}$lder's inequality states that for any $p\geq 1$ and $q\geq 1$ satisfying $\frac{1}{p}+\frac{1}{q}=1$, we have $\mathbb{E}|xy|\leq (\mathbb{E}|x|^p)^{1/p}(\mathbb{E}|y|^q)^{1/q}$.}
\begin{lemma}
\begin{enumerate}[(i)]
\item\,\,  For any random variable $z$,  we have
\begin{equation}\label{eqn: lower_exp_z}
\mathbb{E}|z|^2\leq (\mathbb{E}|z|)^{2/3}\cdot (\mathbb{E}|z|^4)^{1/3}.
\end{equation}
\item\,\,  Assume that $\va=(a_1,\ldots,a_d)\in \mathbb{C}^{d}$ is a random vector, whose elements are independently drawn from some sub-Gaussian random variable $\xi$ with $\|\xi\|_{\psi_2}\leq C_{\psi}$. Then for any fixed $\vz=(z_1,\ldots, z_d)$ with $\|\vz\|_2=1$, we have
\begin{equation}\label{eqn: vector_subGaussian}
\||\langle\va,\vz\rangle|\|_{\psi_2}^2=\left\|\Big|\sum_{l=1}^{d}a_{l}z_l\Big|\right\|_{\psi_2}^2\leq C'\sum_{l=1}^{d} |z_{l}|^2\|a_{l}\|_{\psi_2}^2\leq C'\cdot C_{\psi}^2,
\end{equation}
where $C'$ is an absolute constant.
\end{enumerate}
\end{lemma}
We next introduce the concentration inequality for sub-Gaussian measurements.   A similar result was obtained for the case where the measurements are Gaussian vectors in \cite{XiaXu} (see \cite[Theorem 1.2]{XiaXu}).
\begin{theorem}\label{thm: RIP}
Assume that $\va_j\in \mathbb{C}^d,$  $j=1,\ldots,m,$ are independently taken as complex sub-Gaussian random vectors, whose elements are independently drawn from some sub-Gaussian random variable $\xi$ with $\|\xi\|_{\psi_2}\leq C_{\psi}$, $\mathbb{E}\xi=0$, $\mathbb{E}\xi^2=0$, $\mathbb{E}|\xi|^2=1$, and $\mathbb{E}|\xi|^4=\gamma>1$.
Assume   that $\mathcal{A}: \mathbb{C}^{d\times d}\rightarrow \mathbb{C}^m$ is  the linear operator corresponding to $\{\va_j\}_{j=1}^m$, which is defined in (\ref{eqn: A()}).
If
$
 m\gtrsim k\log (ed/k),
 $
with probability at least $1-2\exp(-c_0m)$,  $\mathcal{A}$
satisfies the restricted isometry property  on the order of $(2,k)$, i.e.,
\[
C_{-}\|\vX\|_{F}\leq\frac{1}{m}\|\mathcal{A}(\vX)\|_{1}\leq C_{+}\|\vX\|_{F}
\]
holds for all Hermitian $\vX\in \mathbb{C}^{d\times d}$ with $\text{rank}(\vX)\leq 2$ and
$\|\vX\|_{0,2}\leq k$ (also $\|\vX^{*}\|_{0,2}\leq k$).
Here, $c_0,$  $C_{-}$ and $C_+$ are positive absolute constants depending on $C_{\psi}$ and $\gamma$.
\end{theorem}

\begin{proof}
 We first consider the lower and upper bound of $\mathbb{E}|\va_j^*\vX\va_j|$, $j=1,\ldots, m$, for  any fixed Hermitian
\[
\vX\in \mathcal{X}:=\{\vX\in \mathbb{H}^{d\times d}\ :\ \|\vX\|_F=1,\ \ \text{rank}(\vX)\leq 2\ \text{and}\ \|\vX\|_{0,2}\leq k\}.
 \]
 For simplicity, we consider the upper and lower bounds of $\mathbb{E}|\va_1^*\vX\va_1|$. We claim that
\begin{equation}\label{eq:ELU}
C_{-}'\leq \inf_{\vX\in\mathcal{X}}\mathbb{E}|\va_1^*X\va_1|\leq \sup_{\vX\in \mathcal{X}}\mathbb{E}|\va_1^*X\va_1|\leq C_{+}',
\end{equation}
for some positive constants $C_{-}'$ and $C_{+}'$, which only depend on $C_{\psi}$ and $\gamma$.
The remaining proof is analog to the proof of \cite[Theorem 1.2]{XiaXu}.  {Therefore, we only provide a sketch of proof for the left part. Based on (\ref{eq:ELU}) and the Berntein-type inequality in \cite[Lemma 2.1]{XiaXu}, we have
\[
\mathbb{P}\left\{\left|\frac{1}{m}\|\mathcal{A}(\boldsymbol{X})\|_1-\frac{1}{m}\mathbb{E}\|\mathcal{A}(\boldsymbol{X})\|_1\right|\geq t\right\}\leq 2\exp\left(-cm\min\{t^2/C_{\psi}^2, t/C_{\psi}\}\right)
\]
for every $t\in (0,1)$. Here $c$ is some positive absolute constant.
Additionally, let $\epsilon\in (0,1)$ and denote $\mathcal{N}_{\epsilon}$ as an $\epsilon$-net of $\mathcal{X}$. The cardinality of $\mathcal{N}_{\epsilon}$ is represented by $\#\mathcal{N}_{\epsilon}$.
Taking the union probability bound, it leads to
\[
C_{-}'-t\leq \frac{1}{m}\|\mathcal{A}(\boldsymbol{X}_0)\|_1\leq C_{+}'+t,\ \text{for}\ \text{all}\ \vX_0\in \mathcal{N}_\epsilon
\]
with probability at least
$
1-2\cdot \#\mathcal{N}_{\epsilon}\cdot \exp\left(-c\min\{1/C_{\psi}^2,1/C_{\psi}\}mt^2\right).
$
Set
\[
U_\A:=\max_{X\in {\mathcal X}} \frac{1}{m}\|\A(X)\|_1.
\]
For any $X\in \mathcal{X}$,  there exits  $X_0\in \mathcal{N}_{\epsilon}$  such that
$\|X-X_0\|_F\leq \epsilon$ and $\|X-X_0\|_{0,2}\leq k$. Following the same proof strategies as  (2.7) in  \cite[Theorem 1.2]{XiaXu}, we have
\begin{equation}\label{eq:upu0}
 \frac{1}{m}\|\mathcal{A}(X)\|_1\leq
\frac{1}{m}\|\mathcal{A}(X_0)\|_1+\frac{1}{m}\|\mathcal{A}(X-X_0)\|_1\leq
C_{+}'+t+\sqrt{2}U_{\mathcal{A}}\epsilon,
\end{equation}
 which leads to
 \[
U_{\mathcal{A}}=\max_{X\in {\mathcal X}} \frac{1}{m}\|\A(X)\|_1\,\, \leq\,\, \frac{C_{+}'+t}{1-\sqrt{2}\epsilon}.
\]
Similarly, we also have
\[
\frac{1}{m}\|\mathcal{A}(X)\|_1\geq \frac{1}{m}\|\mathcal{A}(X_0)\|_1-\frac{1}{m}\|\mathcal{A}(X-X_0)\|_1\geq C_{-}'-t-\sqrt{2}U_{\mathcal{A}}\epsilon
\geq C_{-}'-t-\sqrt{2}\frac{C_{+}'+t}{1-\sqrt{2}\epsilon}\epsilon.
\]
For the remaining part of the proof, we choose $t=\min\{\frac{C_{-}'}{2},1\}$ and $\epsilon=\frac{C_{-}'}{8(C_{+}'+1)}$. Set  $C_{+}:=\frac{C_{+}'+t}{1-\sqrt{2}\epsilon}>0$ and  $C_{-}:=C_{-}'-t-\sqrt{2}\frac{C_{+}'+t}{1-\sqrt{2}\epsilon}\epsilon>0$. Then
\[
C_{-}\|X\|_F\leq
 \frac{1}{m}\|\mathcal{A}(X)\|_1\leq C_{+}\|X\|_F, \text{ for all } X\in \mathcal{X},
\]
with probability at least
\begin{equation}\label{eqn: N_covernumber}
1-2\cdot \#\mathcal{N}_{\epsilon}\cdot \exp\left(-c\min\{1/C_{\psi}^2,1/C_{\psi}\}mt^2\right)\geq 1-2\exp\left(-c'm+c''k\log(ed/ k)\right).
\end{equation}
Here $c'$ and $c''$ are constants depending on $C_{\psi}$ and $\gamma$. The last inequality in (\ref{eqn: N_covernumber}) is based on  $\#\mathcal{N}_{\epsilon}\leq \left(\frac{9\sqrt{2}ed}{\epsilon k}\right)^{4k+2}$  \cite[Lemma 2.3]{XiaXu}. If $m\geq \frac{2c''}{c'} k\log(ed/k)$, the probability,  as calculated using the formula introduced in (\ref{eqn: N_covernumber}), is no less than $1-2\exp(-\frac{c'}{2}m)$. Thus, we  arrives at the conclusion.  }

 It remains to prove the claim (\ref{eq:ELU}). Since $\rank(\vX)\leq 2$ and $\|\vX\|_F=1$, we can write $\vX$ in the form of: \[\vX=\lambda_1\vu_1\vu_1^*+\lambda_2\vu_2\vu_2^*,\] where $\lambda_1,\lambda_2\in \mathbb{R}$ such that $\lambda_1^2+\lambda_2^2=1$ and $\vu_1,\vu_2\in \mathbb{C}^n$ satisfying $\|\vu_1\|_2=\|\vu_2\|_2=1$, $\langle \vu_1,\vu_2\rangle=0$.

On one hand, we can obtain that
\begin{equation}\label{eqn: upper_1}
\mathbb{E}|\va_1^*X\va_1|=\mathbb{E}|\va_1^*(\lambda_1\vu_1\vu_1^*+\lambda_2\vu_2\vu_2^*)\va_1|\\
\leq |\lambda_1|\mathbb{E}|\va_1^*\vu_1|^2+|\lambda_2|\mathbb{E}|\va_1^*\vu_2|^2\leq \sqrt{2}.
\end{equation}

On the other hand, according to (\ref{eqn: lower_exp_z}), we have
\begin{equation}\label{eqn: lower_1}
\mathbb{E}|\va_1^*\vX\va_1|\geq \sqrt{\frac{(\mathbb{E}|\va_1^*\vX\va_1|^2)^3}{\mathbb{E}|\va_1^*\vX\va_1|^4}}.
\end{equation}
In order to get the lower bound of $\mathbb{E}|\va_1^*\vX\va_1|$, we should estimate  the lower bound of $\mathbb{E}|\va_1^*\vX\va_1|^2$ and the upper bound of $\mathbb{E}|\va_1^*\vX\va_1|^4$. By direct calculation, we can get that
\begin{equation}\label{eqn: temp1}
\begin{aligned}
\mathbb{E}|\va_1^*\vX\va_1|^2
=&\mathbb{E}\text{Tr}(\va_1^*\vX\va_1\va_1^*\vX^*\va_1)=\text{Tr}\mathbb{E}(\va_1\va_1^*\vX \va_1\va_1^*\vX^*)\\
=&\|\vX\|_F^2+\left(\sum_{i}^d \vX_{i,i}\right)^2+(\gamma-2)\sum_{i}^{d}|\vX_{i,i}|^2\\
=&(\gamma-1)\|\vX\|_F^2+(2-\gamma)\left(\|\vX\|_F^2-\sum_{i}^{d}|\vX_{i,i}|^2\right)+\left(\sum_{i}^d \vX_{i,i}\right)^2\\
\geq &\min\{1,\gamma-1\}\|\vX\|_F^2=\min\{1,\gamma-1\},
\end{aligned}
\end{equation}
provided that $\gamma>1$. We also have
\begin{equation}\label{eqn: temp2}
\begin{aligned}
\mathbb{E}\abs{\va_1^*\vX\va_1}^4=&\mathbb{E}\abs{\lambda_1\abs{\va_1^*\vu_1}^2+\lambda_2\abs{\va_2^*\vu_2}^2}^4\\
=&\lambda_1^4\mathbb{E}|\va_1^*\vu_1|^8+4\lambda_1^3\lambda_2\mathbb{E}\left(|\va_1^*\vu_1|^6|\va_2^*\vu_2|^2\right)+6\lambda_1^2\lambda_2^2\mathbb{E}\left(|\va_1^*\vu_1|^4|\va_2^*\vu_2|^4\right)\\
&+4\lambda_1\lambda_2^3\mathbb{E}\left(|\va_1^*\vu_1|^2|\va_2^*\vu_2|^6\right)+\lambda_2^4\mathbb{E}|\va_2^*\vu_2|^8\\
\overset{(a)}\leq &(\sqrt{8}\cdot \sqrt{C'}\cdot C_{\psi})^8\cdot (\lambda_1+\lambda_2)^4 \leq 4(\sqrt{8}\cdot \sqrt{C'}\cdot C_{\psi})^8.
\end{aligned}
\end{equation}
{The inequality  (a)   is based on the H$\ddot{\text{o}}$lder's inequality. For instance,
\[
\mathbb{E}\left(|\va_1^*\vu_1|^6|\va_2^*\vu_2|^2\right)\leq (\mathbb{E}\left|\va_1^*\vu_1\right|^8)^{3/4}(\mathbb{E}\left|\va_2^*\vu_2\right|^8)^{1/4}\leq (\sqrt{8}\cdot \sqrt{C'}\cdot C_{\psi})^8.
\]
 The last inequality above follows from the definition of $\|\cdot\|_{\psi_2}$ and (\ref{eqn: vector_subGaussian}), which leads to $\||\va_i^*\vz|\|_{\psi_2}\leq \sqrt{C'}\cdot C_{\psi}$ for any $\|\vz\|_2=1$.}
Substituting  (\ref{eqn: temp1}) and (\ref{eqn: temp2}) into (\ref{eqn: lower_1}), we can obtain that
\begin{equation}\label{eqn: lower_2}
\mathbb{E}|\va_1^*\vX\va_1|\geq {\frac{\min\{1,(\gamma-1)^{3/2}\}}{128\cdot (C')^3\cdot C_\psi^4}}.
\end{equation}
Based on (\ref{eqn: lower_2}) and (\ref{eqn: upper_1}), we  obtain
(\ref{eq:ELU}) holds
with $C_{-}'={\frac{\min\{1,(\gamma-1)^{3/2}\}}{128\cdot (C')^3\cdot C_\psi^4}}$ and $C_{+}'=\sqrt{2}$.
\end{proof}
\begin{remark}
Theorem \ref{thm: RIP} establishes the restricted isometry property on simultaneous low-rank and sparse matrices,  which is similar to the RIP-$\ell_2/\ell_1$ in \cite{CCG15}. However, the RIP-$\ell_2/\ell_1$  in \cite{CCG15} requires the matrix $\vX$ be in the form of low-rank plus sparse, i.e.,   $\vX=\vX_1+\vX_2$, where $\vX_1$ is low rank and $\vX_2$ is sparse \cite[Definition 3]{CCG15}.
  The restricted isometry property on low-rank matrices in real space without sparsity constraint was studied in  \cite{geometry_smooth, T. Cai}. Particularly, they show that
\begin{equation}\label{eqn: general_r}
\|\mathcal{A}(\vX)\|_1\gtrsim 1\qquad \text{and}\qquad \|\mathcal{A}(\vX)\|_1\lesssim \sqrt{r},
\end{equation}
for all $\vX\in \mathbb{R}^{d\times d}$ with $\text{rank}(\vX)\leq r$ and $\|\vX\|_F=1$.
\end{remark}

Theorem \ref{thm: SRIP} is a strong version of Theorem \ref{thm: RIP}, which plays an important role to extend the results in the real case into the complex case.  
\begin{theorem}\label{thm: SRIP}
Under the assumption of Theorem \ref{thm: RIP},
there exists a universal constant $\beta_0>0$ such that, with probability at least $1-2 \exp(-c_0m/4)$,  the following holds:
\begin{equation}\label{eq:strong}
C_{-}(1-\beta_0)\|\vX\|_{F}\leq\frac{1}{m}\|\mathcal{A}_I(\vX)\|_{1}\leq {C_+}\|\vX\|_{F},\quad \text{for all } I\subset\{1,\ldots,m\},\text{ with }\#I \geq (1-\beta_0)m,
\end{equation}
for all Hermitian $\vX\in \mathbb{C}^{d\times d}$ with $\text{rank}(\vX)\leq 2$ and
$\|\vX\|_{0,2}\leq k$ (also $\|\vX^{*}\|_{0,2}\leq k$), provided
\[
 m\gtrsim \frac{1}{1-\beta_0}k\log (ed/k).
 \]
 Here, $c_0,$  $C_{-},$ $C_+$ are positive constants presented in Theorem \ref{thm: RIP}.
\end{theorem}
\begin{proof}
According to Theorem \ref{thm: RIP}, for any fixed $I_0\subset \{1,\ldots,m\}$ with $\#I_0\geq (1-\beta_0)m$,
\[
C_-\|\vX\|_{F}\leq\frac{1}{\#I_0}\|\mathcal{A}_{I_0}(\vX)\|_{1}\leq C_+\|\vX\|_{F}
\]
holds with probability at least $ 1-2\exp(-c_0(1-\beta_0)m)$, provided  $(1-\beta_0)m\gtrsim  k\log(ed/k)$. A simple observation is that
\[
C_-(1-\beta_0)\|\vX\|_{F}\leq\frac{1}{m}\|\mathcal{A}_{I_0}(\vX)\|_{1}= \frac{\#I_0}{m}\cdot\frac{1}{\#I_0}\|\mathcal{A}_{I_0}(\vX)\|_{1}\leq {C_+}\|\vX\|_{F}.
\]
 {If $\beta_0<1/2$, taking the union probability bound,} (\ref{eq:strong}) holds with  probability at least
\begin{equation*}
\begin{aligned}
1-2\sum_{\# I\geq (1-\beta_0)m}{m\choose \#I}\exp(-c_0(1-\beta_0)m)&\geq 1-2\sum_{\# I\geq (1-\beta_0)m}{m\choose (1-\beta_0)m}\exp(-c_0(1-\beta_0)m)\\
&=1-2(\beta_0m+1){m\choose (1-\beta_0)m}\exp(-c_0(1-\beta_0)m)\\
&\geq 1- 2\exp(\beta_0m+\beta_0\log(e/\beta_0)m-c_0(1-\beta_0)m).
\end{aligned}
\end{equation*}
The last inequality follows from $1+\beta_0m\leq \exp(\beta_0m)$, and
 \[
 {m\choose (1-\beta_0)m}={m\choose \beta_0m}\leq (\frac{e}{\beta_0})^{\beta_0m}=\exp(\beta_0\log (e/\beta_0)m).
 \]
 We can take $\beta_0>0$ small enough so that
  \begin{equation}\label{eqn: beta_condition}
  \beta_0\log(e/\beta_0)\leq \frac{c_0}{4}, \qquad\beta_0\leq \frac{c_0}{4},\qquad\text{and}\qquad\beta_0\leq\frac{1}{4}.
  \end{equation}
   Hence, the inequality (\ref{eq:strong}) holds with probability at least
\[
1- 2\exp(\beta_0m+\beta_0\log(e/\beta_0)m-c_0(1-\beta_0)m)\geq 1-2 \exp(-c_0m/4) .
\]
We arrive at the conclusion.
\end{proof}

\begin{remark}
Theorem \ref{thm: SRIP} claims that there exits $\beta_0>0$ so that (\ref{eq:strong}) holds.  In fact, according to (\ref{eqn: beta_condition}),
we can take $\beta_0:=\min\{\frac{c_0}{8}, \exp\left(-\frac{c_0}{8}\right), \frac{1}{4}\}$, where $c_0$ is defined in Theorem \ref{thm: RIP}.
 If we take $k=d$ in Theorem \ref{thm: SRIP}, then (\ref{eq:strong}) holds with high probability provided  $m\gtrsim d$.
\end{remark}
The following lemma plays a key role in the proof of Theorem   \ref{th:upbound}. We postpone its proof to the end of this section.
\begin{lemma}\label{lem: non_sparse}
Assume that
 $\va_j,$ $j=1,\ldots,m,$ are independently taken as complex sub-Gaussian random vectors, whose elements are independently drawn from some sub-Gaussian random variable $\xi$ with $\|\xi\|_{\psi_2}\leq C_{\psi}$, $\mathbb{E}\xi=0$, $\mathbb{E}\xi^2=0$, $\mathbb{E}|\xi|^2=1$, and $\mathbb{E}|\xi|^4=\gamma>1$. For all $\vx_0\in \mathbb{C}^d$, assume that $\widehat{\vx}\in \mathbb{C}^d$ is the solution to  (\ref{eqn: nonsparse_model}) with
$\mathbb{F}=\mathbb{C}$ and $b_j=\abs{\innerp{\va_j,\vx_0}}+\eta_j$, $j=1,\ldots,m$.
If $m\gtrsim d$, the following holds with probability at least $1-2 \exp(-c_0m/4)$:
\begin{equation}\label{eqn: non_sparse_result1}
\|\mathcal{A}_I({ \widehat{\vx}\widehat{\vx}^{*}-\vx_0\vx_0^*})\|_1\lesssim\sqrt{m}\|\ve\|_2(\|\widehat{\vx}\|_2+\|\vx_0\|_2).
\end{equation}
Here, $c_0$ is the positive constant which is defined in Theorem \ref{thm: RIP}.
The index set $I:=I_{\hat{\vx},\vx_0}\subset\{1,\ldots,m\}$ is chosen by the following way. Set $\xi_j:=\abs{\innerp{\va_j,\widehat{\vx}}}+\abs{\innerp{\va_j,\vx_0}}$ and assume that $\xi_{j_1}\leq \xi_{j_2}\leq \cdots\leq \xi_{j_m}$. Take $I=\{j_t\ |\ t\leq (1-\beta_0)m\}$, where $\beta_0>0$ is defined in Theorem \ref{thm: SRIP}.
\end{lemma}

We next present the proof of Theorem  \ref{th:upbound}.
\begin{proof}[Proof of Theorem \ref{th:upbound}]
If $ \widehat{\vx}=\vx_0=0$, the conclusion holds. We next assume that either $\widehat{\vx}\neq 0$ or $\vx_0\neq 0$.
Suppose that $\theta_0\in \R$ such that   $\langle \widehat{\vx}, e^{ \mathrm{i} \theta_0}\vx_0\rangle\geq 0$. According to
 Lemma \ref{Lemma 3.2},  we have
\[
\|\widehat{\vx}\widehat{\vx}^{*}-\vx_0\vx_0^*\|_F^2\geq \frac{1}{2}\cdot \|\widehat{\vx}\|_2^2 \cdot \|\widehat{\vx}-e^{\mathrm{i}\theta_0}\vx_0\|_2^2 \qquad \text{and}\qquad \|\widehat{\vx}\widehat{\vx}^{*}-\vx_0\vx_0^*\|_F^2\geq \frac{1}{2}\cdot \|\vx_0\|_2^2 \cdot \|\widehat{\vx}-e^{\mathrm{i}\theta_0}\vx_0\|_2^2,
\]
which implies  that
\begin{equation}\label{eqn: matrix_vector_remark}
\|\widehat{\vx}\widehat{\vx}^{*}-\vx_0\vx_0^*\|_F^2\geq \frac{1}{4}\left( \|\widehat{\vx}\|_2^2+ \|\vx_0\|_2^2\right)\|\widehat{\vx}-e^{\mathrm{i}\theta_0}\vx_0\|_2^2\geq \frac{1}{8}(\|\widehat{\vx}\|_2+\|\vx_0\|_2)^2\|\widehat{\vx}-e^{\mathrm{i}\theta_0}\vx_0\|_2^2.
\end{equation}
Hence, we have
\begin{equation}\label{eq:sqrt2}
\|{ \widehat{\vx}\widehat{\vx}^{*}-\vx_0\vx_0^*}\|_F \geq \frac{\sqrt{2}}{4}(\|\widehat{\vx}\|_2+\|\vx_0\|_2) \|\widehat{\vx}-e^{\mathrm{i}\theta_0}\vx_0\|_2\geq  \frac{\sqrt{2}}{4}(\|\widehat{\vx}\|_2+\|\vx_0\|_2)\inf_{\theta\in[0,2\pi)}\|\widehat{\vx}-e^{\mathrm{i}\theta}\vx_0\|_2.
\end{equation}
Combining  (\ref{eq:sqrt2}) and Theorem \ref{thm: SRIP} with $k=d$, we obtain that
\begin{equation*}
\begin{aligned}
&\frac{\sqrt{2}C_{-}}{4}(\|\widehat{\vx}\|_2+\|\vx_0\|_2)\inf_{\theta\in[0,2\pi)}\|\widehat{\vx}-\exp(\mathrm{i}\theta)\vx_0\|_2
\leq C_{-}\|{ \widehat{\vx}\widehat{\vx}^{*}-\vx_0\vx_0^*}\|_F\\
&\qquad \leq \frac{1}{(1-\beta_0)m}\|\mathcal{A}_I({ \widehat{\vx}\widehat{\vx}^{*}-\vx_0\vx_0^*})\|_1\lesssim\frac{1}{\sqrt{m}}\|\ve\|_2(\|\widehat{\vx}\|_2+\|\vx_0\|_2),
\end{aligned}
\end{equation*}
which implies the conclusion. Here, the set $I$ is defined in Lemma \ref{lem: non_sparse} and the last inequality follows from (\ref{eqn: non_sparse_result1}).
\end{proof}

\begin{proof}[Proof of Lemma \ref{lem: non_sparse}]
If $\widehat{\vx}=\vx_0=0$, then the conclusion holds.  We next assume that either $\widehat{\vx}\neq 0$ or $\vx_0\neq 0$.
According to the definition of $\widehat{\vx}$, we have
\[
\sum_{j=1}^m(\abs{\innerp{\va_j,\widehat{\vx}}}-b_j)^2\leq \sum_{j=1}^m(\abs{\innerp{\va_j,\vx_0}}-b_j)^2.
\]
Thus
\begin{equation}\label{eqn: error_lower_0}
\sum_{j=1}^m \left|\abs{\innerp{\va_j,\widehat{\vx}}}-b_j\right|\leq \sqrt{m}\sqrt{\sum_{j=1}^m(\abs{\innerp{\va_j,\widehat{\vx}}}-b_j)^2}\leq \sqrt{m}\sqrt{ \sum_{j=1}^m(\abs{\innerp{\va_j,\vx_0}}-b_j)^2}\leq \sqrt{m}\|\ve\|_2.
\end{equation}
Furthermore, we have
\begin{equation}\label{eqn: error_lower_1}
\begin{aligned}
\sum_{j=1}^m \left|\abs{\innerp{\va_j,\widehat{\vx}}}-b_j\right|&=\sum_{j=1}^m \left|\abs{\innerp{\va_j,\widehat{\vx}}}-\abs{\innerp{\va_j,\vx_0}}-\eta_j\right|\\
&=\sum_{j=1}^m\frac{\left|\abs{\innerp{\va_j,\widehat{\vx}}}^2-\abs{\innerp{\va_j,\vx_0}}^2-\eta_j(\abs{\innerp{\va_j,\widehat{\vx}}}+\abs{\innerp{\va_j,\vx_0}})\right|}{\abs{\innerp{\va_j,\widehat{\vx}}}+\abs{\innerp{\va_j,\vx_0}}}\\
&\geq \sum_{j\in I}\frac{\left|\left|\innerp{\va_j\va_j^*, \widehat{\vx}\widehat{\vx}^{*}-\vx_0\vx_0^*}\right|-|\eta_j|(\abs{\innerp{\va_j,\widehat{\vx}}}+\abs{\innerp{\va_j,\vx_0}})\right|}{\abs{\innerp{\va_j,\widehat{\vx}}}+\abs{\innerp{\va_j,\vx_0}}}.
\end{aligned}
\end{equation}
Combining (\ref{eqn: error_lower_0}) and (\ref{eqn: error_lower_1}), we obtain that
\begin{equation}\label{eqn: commonly_used1}
\sqrt{m}\|\ve\|_2\geq \sum_{j\in I}\frac{\big|\left|\innerp{\va_j\va_j^*, \widehat{\vx}\widehat{\vx}^{*}-\vx_0\vx_0^*}\right|-|\eta_j|(\abs{\innerp{\va_j,\widehat{\vx}}}+\abs{\innerp{\va_j,\vx_0}})\big|}{\abs{\innerp{\va_j,\widehat{\vx}}}+\abs{\innerp{\va_j,\vx_0}}}.
\end{equation}
Then we have
\begin{equation}\label{eqn: error_lower_final}
\begin{split}
\sqrt{m}\|\ve\|_2&\geq \sum_{j\in I}\frac{\left|\left|\innerp{\va_j\va_j^*, \widehat{\vx}\widehat{\vx}^{*}-\vx_0\vx_0^*}\right|-|\eta_j|(\abs{\innerp{\va_j,\widehat{\vx}}}+\abs{\innerp{\va_j,\vx_0}})\right|}{\abs{\innerp{\va_j,\widehat{\vx}}}+\abs{\innerp{\va_j,\vx_0}}}\\
&\geq   \sum_{j\in I}\frac{\left|\left|\innerp{\va_j\va_j^*, \widehat{\vx}\widehat{\vx}^{*}-\vx_0\vx_0^*}\right|-|\eta_j|(\abs{\innerp{\va_j,\widehat{\vx}}}+\abs{\innerp{\va_j,\vx_0}})\right|}{\sqrt{2C_{\beta_0}}(\|\vx_0\|_2+\|\widehat{\vx}\|_2)}\\
&\geq  \frac{\left|\sum_{j\in I}\left|\innerp{\va_j\va_j^*, \widehat{\vx}\widehat{\vx}^{*}-\vx_0\vx_0^*}\right|-\sum_{j\in I}|\eta_j|(\abs{\innerp{\va_j,\widehat{\vx}}}+\abs{\innerp{\va_j,\vx_0}})\right|}{\sqrt{2C_{\beta_0}}(\|\vx_0\|_2+\|\widehat{\vx}\|_2)}.
\end{split}
\end{equation}
where $C_{\beta_0}:=\frac{C_+}{\beta_0}$ and the second inequality in (\ref{eqn: error_lower_final}) follows from
\begin{equation}\label{eq:vxup}.
\abs{\innerp{\va_j,\vx_0}}+\abs{\innerp{\va_j,\widehat{\vx}}}
\leq  \sqrt{2C_{\beta_0}}(\|\vx_0\|_2+\|\widehat{\vx}\|_2),\quad \text{for } j\in I.
\end{equation}
We postpone the argument of (\ref{eq:vxup}) until the end of this proof.
We next employ  (\ref{eqn: error_lower_final}) to show that
\begin{equation}\label{eqn: error_lower_final1}
\|\mathcal{A}_{I}(\widehat{\vx}\widehat{\vx}^*-\vx_0\vx_0^*)\|_1\lesssim \sqrt{m}\|\ve\|_2(\|\widehat{\vx}\|_2+\|\vx_0\|_2)
\end{equation}
under the following two cases.

{\bf{Case 1:}} We assume that the following holds:
\begin{equation}\label{eqn: case1}
\sum_{j\in I}\left|\innerp{\va_j\va_j^*, \widehat{\vx}\widehat{\vx}^{*}-\vx_0\vx_0^*}\right|\geq 4\sum_{j\in I}|\eta_j|(\abs{\innerp{\va_j,\widehat{\vx}}}+\abs{\innerp{\va_j,\vx_0}}).
\end{equation}
Then we have
{
\begin{equation}\label{eqn: case1_temp}
\begin{aligned}
&\left|\sum_{j\in I}\left|\innerp{\va_j\va_j^*, \widehat{\vx}\widehat{\vx}^{*}-\vx_0\vx_0^*}\right|-\sum_{j\in I}|\eta_j|(\abs{\innerp{\va_j,\widehat{\vx}}}+\abs{\innerp{\va_j,\vx_0}})\right|^2\\
=&\left(\sum_{j\in I}\left|\innerp{\va_j\va_j^*, \widehat{\vx}\widehat{\vx}^{*}-\vx_0\vx_0^*}\right|\right)^2+\left(\sum_{j\in I}|\eta_j|(\abs{\innerp{\va_j,\widehat{\vx}}}+\abs{\innerp{\va_j,\vx_0}})\right)^2\\
&-2\left(\sum_{j\in I}\left|\innerp{\va_j\va_j^*, \widehat{\vx}\widehat{\vx}^{*}-\vx_0\vx_0^*}\right|\right)\left(\sum_{j\in I}|\eta_j|(\abs{\innerp{\va_j,\widehat{\vx}}}+\abs{\innerp{\va_j,\vx_0}})\right)\\
\geq &\left(\sum_{j\in I}\left|\innerp{\va_j\va_j^*, \widehat{\vx}\widehat{\vx}^{*}-\vx_0\vx_0^*}\right|\right)^2-2\left(\sum_{j\in I}\left|\innerp{\va_j\va_j^*, \widehat{\vx}\widehat{\vx}^{*}-\vx_0\vx_0^*}\right|\right)\left(\sum_{j\in I}|\eta_j|(\abs{\innerp{\va_j,\widehat{\vx}}}+\abs{\innerp{\va_j,\vx_0}})\right)\\
\geq & \left(\sum_{j\in I}\left|\innerp{\va_j\va_j^*, \widehat{\vx}\widehat{\vx}^{*}-\vx_0\vx_0^*}\right|\right)^2-\frac{1}{2}\left(\sum_{j\in I}\left|\innerp{\va_j\va_j^*, \widehat{\vx}\widehat{\vx}^{*}-\vx_0\vx_0^*}\right|\right)^2\\
=&\frac{1}{2}\left(\sum_{j\in I}\left|\innerp{\va_j\va_j^*, \widehat{\vx}\widehat{\vx}^{*}-\vx_0\vx_0^*}\right|\right)^2=\frac{1}{2}\|\mathcal{A}_I({\widehat{\vx}\widehat{\vx}^{*}-\vx_0\vx_0^*})\|_1^2,
\end{aligned}
\end{equation}
}
where the second inequality follows from  (\ref{eqn: case1}).
Combining  (\ref{eqn: case1_temp}) and (\ref{eqn: error_lower_final}), we obtain that
\[
\|\mathcal{A}_I({ \widehat{\vx}\widehat{\vx}^{*}-\vx_0\vx_0^*})\|_1\leq 2\sqrt{C_{\beta_0}}\sqrt{m}\|\ve\|_2(\|\widehat{\vx}\|_2+\|\vx_0\|_2),
\]
which implies the conclusion.

{\bf{Case 2:}}
We next consider the case where
\begin{equation}\label{eqn: case2}
\sum_{j\in I}\left|\innerp{\va_j\va_j^*, \widehat{\vx}\widehat{\vx}^{*}-\vx_0\vx_0^*}\right|< 4\sum_{j\in I}|\eta_j|(\abs{\innerp{\va_j,\widehat{\vx}}}+\abs{\innerp{\va_j,\vx_0}}).
\end{equation}
According to (\ref{eqn: case2}), we have
\begin{equation*}
\begin{aligned}
\|\mathcal{A}_I({ \widehat{\vx}\widehat{\vx}^{*}-\vx_0\vx_0^*})\|_1&\leq 4\|\ve\|_2\sqrt{\sum_{j\in I}(\abs{\innerp{\va_j,\widehat{\vx}}}+\abs{\innerp{\va_j,\vx_0}})^2}\\
&\leq 4\sqrt{2C_{\beta_0}}\sqrt{m}\|\ve\|_2(\|\vx_0\|_2+\|\widehat{\vx}\|_2).
\end{aligned}
\end{equation*}
The last inequality above follows from (\ref{eq:vxup}). We arrive at the conclusion.

It remains to prove (\ref{eq:vxup}).
A simple observation  is that, for $j\in I$,
\begin{equation*}\label{eqn: error_lower_2}
\begin{aligned}
\left(\abs{\innerp{\va_j,\vx_0}}+\abs{\innerp{\va_j,\widehat{\vx}}}\right)^2& {\leq \frac{1}{\beta_0m}\sum_{t\in \{1,\ldots,m\}\setminus I } (\abs{\innerp{\va_t,\vx_0}}+\abs{\innerp{\va_t,\widehat{\vx}}})^2}\\
&\leq \frac{2}{\beta_0m}\sum_{t\in \{1,\ldots,m\}\setminus I } (\abs{\innerp{\va_t,\vx_0}}^2+\abs{\innerp{\va_t,\widehat{\vx}}}^2)\\
&=\frac{2}{\beta_0m}\|\mathcal{A}_{\{1,\ldots,m\}\setminus I}(\vx_0\vx_0^*+\widehat{\vx}\widehat{\vx}^*)\|_1\\
&\leq \frac{2}{\beta_0}\cdot\frac{1}{m}\|\mathcal{A}(\vx_0\vx_0^*+\widehat{\vx}\widehat{\vx}^*)\|_1\\
&\overset{(a)}\leq \frac{2C_+}{\beta_0}\|\vx_0\vx_0^*+\widehat{\vx}\widehat{\vx}^*\|_F\\
&\leq 2C_{\beta_0} (\|\vx_0\|_2^2+\|\widehat{\vx}\|_2^2)
\leq  2C_{\beta_0}(\|\vx_0\|_2+\|\widehat{\vx}\|_2)^2,
\end{aligned}
\end{equation*}
which implies (\ref{eq:vxup}). Here, inequality (a) follows from
 Theorem \ref{thm: RIP}.
\end{proof}

\section{Proof of Theorem \ref{thm: sharpness}}
Before presenting the proof of Theorem \ref{thm: sharpness}, we introduce the following lemmas.
\begin{lemma}(Hoeffding-type inequality) (\cite[Theorem 2.6.3]{Vershynin})\label{lem: hoeffding}
Let $x_1,\ldots,x_N$ be independent centered sub-Gaussian random variables, and let $K=\max_{l}\|x_{l}\|_{\psi_2}$. Then, for every $\va=(a_1,\ldots,a_N)\in \mathbb{R}^N$ and every $t\geq 0$, we have
\[
\mathbb{P}\left\{\left|\sum_{l=1}^N a_{l}x_{l}\right|\geq t\right\}\leq 2\exp\left(-\frac{c't^2}{K^2\|\va\|_2^2}\right),
\]
where $c'>0$ is an absolute constant.
\end{lemma}
\begin{theorem}(Norm of matrices with sub-Gaussian entries)(\cite[Theorem 4.4.5]{Vershynin})\label{thm: upper_A} Let $\vA$ be an $m\times d$ random matrix whose entries $A_{j,l}$ are independent, mean zero, sub-Gaussian random variables. Then, for any $t>0$, the following holds
with probability at least $1-2\exp(-t^2)$:
\[
\|\vA\|\leq C'K(\sqrt{m}+\sqrt{d}+t).
\]
 Here $K=\max_{j,l}\|A_{j,l}\|_{\psi_2}$ and $C'>0$ is an absolute constant.
\end{theorem}

\begin{lemma}\label{lem: lower}
Assume that $\va_j,$  $j=1,\ldots,m,$  are
i.i.d. complex sub-Gaussian random vectors,  whose elements are independently drawn from some sub-Gaussian random variable $\xi$ with $\|\xi\|_{\psi_2}\leq C_{\psi}$, $\mathbb{E}\xi=0$, $\mathbb{E}|\xi|^2=1$, and $\mathbb{E}|\xi|^4=\gamma>1$. Assume that $\boldsymbol{\ve}=(\eta_1,\dots,\eta_m)\in \R^m$ satisfies
 $\abs{\sum_{j=1}^m\eta_j}\geq C \sqrt{m}\|\boldsymbol{\ve}\|_2$ with $C\in (0,1)$.
Then   the following holds with probability at least $1-4\exp\left(-{c}m\right)$:
\[
\Bigg|\sum_{j=1}^m\eta_j\left|\langle \va_j,\vx\rangle\right|\Bigg|\gtrsim \sqrt{m}\|\ve\|_2\|\vx\|_2\quad \text{for all}\  \vx\in \mathbb{C}^d,
\]
provided   $m\gtrsim d$. Here, $c$ is a positive constant depending on $C_{\psi}$ and $\gamma$.
\end{lemma}
\begin{proof}
First of all, we consider an upper bound of $\left|\sum_{j=1}^m\left|\eta_j\right|\cdot\left|\langle \va_j,\vx\rangle\right|\right|$. We can obtain that the following holds for all $\vx\in \mathbb{C}^d$ with probability at least $1-2\exp(-m)$:
 \begin{equation}\label{eqn: temp_upper0}
 \begin{aligned}
 \left|\sum_{j=1}^m\left|\eta_j\right|\cdot\left|\langle \va_j,\vx\rangle\right|\right|
 \leq \|\ve\|_2 \cdot \left(\sum_{j=1}^m \abs{\innerp{\va_j,\vx}}^2\right)^{1/2}
 \leq 3\cdot C_\psi\cdot C'\cdot  \sqrt{m}\|\ve\|_2\|\vx\|_2,  \end{aligned}
\end{equation}
where $C'>0$ is the positive absolute constant defined in Theorem \ref{thm: upper_A}. Here, the second inequality follows from Theorem \ref{thm: upper_A} with $t=\sqrt{m}$ and $m\geq C_2d$ for some $C_2\geq 1$, which will be chosen later. 

Then we begin to  consider   the lower bound of  $\left|\sum_{j=1}^m\eta_j\left|\langle \va_j,\vx\rangle\right|\right|$.  { For convenience, assume that $\|\vx\|_2=1$.} When $\vx$ is fixed, we can use (\ref{eqn: vector_subGaussian}) to obtain that $\||\langle\va_1, \vx\rangle|\|_{\psi_2}\leq \sqrt{C'}C_{\psi}$. Applying Lemma \ref{lem: hoeffding}, we have


 \begin{equation}\label{eq:lowergauss}
 \mathbb{P}\left\{\left|\sum_{j=1}^m \eta_j|\langle \va_{j},\vx\rangle|-\sum_{j=1}^m\eta_j\mathbb{E}|\langle \va_{j},\vx\rangle|\right|\geq t\right\}\leq 2 \exp\left(-\frac{c't^2}{C'C_{\psi}^{2}\|\boldsymbol{\ve}\|_2^2}\right), \quad \text{ for any } t\geq 0.
 \end{equation}
 Taking $t=C_1\sqrt{m}\|\boldsymbol{\ve}\|_2$ in (\ref{eq:lowergauss}) for some $C_1>0$, we obtain that
the following holds with probability at least $1-2 \exp\left(-\frac{c'C_1^2}{C'C_{\psi}^2}m\right)$:
 \begin{equation}\label{eqn: eta_lower_temp1}
 \begin{aligned}
 \left|\sum_{j=1}^m \eta_j|\langle \va_j,\vx\rangle|\right|& \geq
 \left|\sum_{j=1}^m\ve_{j}\mathbb{E}|\langle \va_{j},\vx\rangle|\right|-C_1\sqrt{m}\|\boldsymbol{\ve}\|_2\\
 &\overset{(a)}\geq C_0\left|\sum_{j=1}^m\ve_{j}\right|-C_1\sqrt{m}\|\boldsymbol{\ve}\|_2\\
 &\overset{(b)}\geq \left(C_0\cdot C\cdot \sqrt{m}-C_1\sqrt{m}\right)\|\boldsymbol{\ve}\|_2 \geq \left(C_0\cdot C-C_1\right)\sqrt{m}\|\ve\|_2,
 \end{aligned}
 \end{equation}
 The inequality $(a)$ can be derived using the following expression:
 \begin{equation}
 \mathbb{E}|\langle \va_{j},\vx\rangle| \geq \sqrt{\frac{ (\mathbb{E}|\langle \va_{j},\vx\rangle|^2)^3}{ \mathbb{E}|\langle \va_{j},\vx\rangle|^4}}\geq \sqrt{\frac{1}{\gamma}}=:C_0,
 \end{equation}
 where we use (\ref{eqn: lower_exp_z}) and the definition of $\|\cdot\|_{\psi_2}$.
 The inequality $(b)$ is based on $\left|\sum_{j=1}^m\eta_j\right|\geq C \sqrt{m}\|\boldsymbol{\ve}\|_2$. {The constant  $C_1$ is chosen such  that  $C_0\cdot C>C_1$, with the specific value of $C_1$ to be determined later. }

Assume that $\mathcal{N}$ is an $\epsilon$-net of the unit complex sphere in $\mathbb{C}^d$ and hence the covering number $\#\mathcal{N}\leq (1+\frac{2}{\epsilon})^{2d}$.
 By the union bound, we obtain that
 \begin{equation}\label{eqn: cover_uniform_up2}
\left|\sum_{j=1}^m\eta_j|\langle \va_j,\vx\rangle|\right|\geq \left(C_0\cdot C-C_1\right) \sqrt{m}\|\boldsymbol{\ve}\|_2\qquad \text { for any }\vx\in \mathcal{N}
 \end{equation}
  holds with probability at least \[1-2 \exp\left(-\frac{c'C_1^2}{C'C_{\psi}^2}m+2\log(1+2/\epsilon)d\right)\geq 1-2 \exp\left(-\left(\frac{c'C_1^2}{C'C_{\psi}^2}-\frac{2\log(1+2/\epsilon)}{C_2}\right)m\right).\]
For any $\vx\in \mathbb{C}^d$ with $\|\vx\|_2=1$, there exists some $\vx'\in \mathcal{N}$ such that $\|\vx'-\vx\|_2\leq \epsilon$. Then
 \begin{equation}\label{eqn: cover_uniform_up1}
 \begin{aligned}
\left| \Bigg|\sum_{j=1}^m\eta_j\left|\langle \va_j, \vx'\rangle\right|\Bigg|-\Bigg|\sum_{j=1}^m\eta_j\left|\langle \va_j, \vx\rangle\right|\Bigg|\right|\leq& \sum_{j=1}^m|\eta_j|\cdot|\langle\va_j, \vx'-\vx\rangle|\\
=&\|\vx'-\vx\|_2\sum_{j=1}^m|\eta_j|\cdot\left|\left\langle\va_j, \frac{\vx'-\vx}{\|\vx'-\vx\|_2}\right\rangle\right|\\
\leq & 3 \cdot \epsilon\cdot C_\psi\cdot C'\cdot \sqrt{m}\|\boldsymbol{\ve}\|_2.
\end{aligned}
 \end{equation}
 The last line above is based on (\ref{eqn: temp_upper0}).
 Combining  (\ref{eqn: cover_uniform_up2}) and (\ref{eqn: cover_uniform_up1}), we obtain that
\begin{equation}
\begin{aligned}
 \min_{\|\vx\|=1} \left|\sum_{j=1}^m\eta_j\abs{\innerp{\va_j,\vx}}\right|\geq& \left(C_0\cdot C-C_1\right) \sqrt{m}\|\boldsymbol{\ve}\|_2- 3 \cdot \epsilon\cdot C_\psi\cdot C' \cdot\sqrt{m}\|\boldsymbol{\ve}\|_2\\
=&\left(C_0\cdot C-C_1- 3 \cdot \epsilon\cdot C_\psi\cdot C'\right)\sqrt{m}\|\boldsymbol{\ve}\|_2,
\end{aligned}
\end{equation}
holds with probability at least $1-2 \exp\left(-(\frac{c'C_1^2}{C'C_{\psi}^2}-\frac{2\log(1+2/\epsilon))}{C_2})m\right)-2\exp(-m).$ Therefore, choosing   $\epsilon$ satisfying $3\cdot \epsilon\cdot C_\psi\cdot C' \leq \frac{C_0\cdot C}{4}$, $C_1$ satisfying $C_1<\frac{C_0\cdot C}{4}$, and $C_2$ satisfying $\frac{c'C_1^2}{4C'C_{\psi}^2}\geq \frac{\log(1+2/\epsilon))m}{C_2}$, we  obtain that
\[
\min_{\|\vx\|_2=1}\left|\sum_{j=1}^m\eta_j\left|\langle \va_j, \vx\rangle\right|\right| \geq \frac{C_0\cdot C}{4} \sqrt{m}\|\boldsymbol{\ve}\|_2,
\]
holds with probability at least $1-2 \exp\left(-\frac{1}{2}\frac{c'C_1^2}{C'C_{\psi}^2}m\right)-2\exp(-m)\geq 1-4\exp(-{c}m)$, where ${c}=\min\left\{1,\frac{c'C_1^2}{2C'C_{\psi}^2}\right\}>0$.
\end{proof}
\begin{lemma}\label{lem: upper}
Assume that $\va_j,$ $j=1,\ldots,m,$  are i.i.d. complex sub-Gaussian random vectors, whose elements are independently drawn from some sub-Gaussian random variable $\xi$ with $\|\xi\|_{\psi_2}\leq C_{\psi}$, $\mathbb{E}\xi=0$, $\mathbb{E}\xi^2=0$, $\mathbb{E}|\xi|^2=1$, and $\mathbb{E}|\xi|^4=\gamma>1$.
Assume that $m\gtrsim d$.
Then the following holds with probability at least $1-2 \exp\left(-m\right)$:
\[
\sum_{j=1}^m|\langle \va_j,\boldsymbol{u}\rangle|\cdot|\langle \va_j,\boldsymbol{v}\rangle|\,\,\lesssim\,\, m\|\vu\|_2 \|\vv\|_2,\quad \text{ for all } \boldsymbol{u},\boldsymbol{v}\in \mathbb{C}^d.
\]
\end{lemma}
\begin{proof}
 Taking  $t=\sqrt{m}$ in Theorem \ref{thm: upper_A}, we can   obtain that
\begin{equation*}
\sum_{j=1}^m\abs{\langle \va_j,\boldsymbol{u}\rangle|\cdot|\langle \va_j,\boldsymbol{v}\rangle} \leq (\sum_{j=1}^m \abs{\innerp{\va_j,\vu}}^2)^{1/2}
 (\sum_{j=1}^m \abs{\innerp{\va_j,\vv}}^2)^{1/2} \lesssim\,\, m\|\vu\|_2 \|\vv\|_2,
\end{equation*}
holds with probability at least $1-2 \exp\left(-m\right)$. Here, we use $m \gtrsim d$.
\end{proof}
We next present the proof of  Theorem \ref{thm: sharpness}.
\begin{proof}[Proof of Theorem \ref{thm: sharpness}]
Recall that $ \widehat{\vx}$ is a solution to $\min_{\vx\in \mathbb{C}^d} \sum_{j=1}^m\left(|\langle \va_j,\vx\rangle|-b_j\right)^2$.
For the case where $\widehat{\vx}=\boldsymbol{0}$, the conclusion follows from $\|\vx_0\|_2\gtrsim \frac{\|\ve\|_2}{\sqrt{m}}$. We next assume that $\widehat{\vx}\neq \boldsymbol{0}$.
We claim that
\begin{equation}\label{eqn: subgradient_equation}
\sum_{j=1}^{m}(|\langle\va_j,\widehat{\vx}\rangle|-b_j)|\langle\va_j,\widehat{\vx}\rangle|\,\,=\,\,0
\end{equation}
and postpone its argument until the end of the proof. Substituting  $b_j=\abs{\innerp{\va_j,\vx_0}}+\eta_j,$ $j=1,\ldots,m$ into   (\ref{eqn: subgradient_equation}), we obtain that
\[
\sum_{j=1}^m\eta_j|\langle\va_j,\widehat{\vx}\rangle|=\sum_{j=1}^{m}\left(|\langle\va_j,\widehat{\vx}\rangle|^2-|\langle\va_j,\vx_0\rangle|\cdot |\langle\va_j,\widehat{\vx}\rangle|\right),
\]
which implies
\begin{equation}\label{eq:bb}
\left|\sum_{j=1}^m\eta_j|\langle\va_j,\widehat{\vx}\rangle|\right|\leq   \sum_{j=1}^{m}|\langle\va_j,\widehat{\vx}\rangle|\cdot \big||\langle\va_j,\widehat{\vx}\rangle|-|\langle\va_j,\vx_0\rangle|\big|\leq \sum_{j=1}^{m}|\langle\va_j,\widehat{\vx}\rangle|\cdot \left|\langle\va_j,\widehat{\vx}-\exp(\mathrm{i}\theta)\vx_0\rangle\right|
\end{equation}
for any $\theta\in [0,2\pi)$.
We use (\ref{eq:bb}) to obtain that
\begin{equation}\label{eq:bb1}
\sqrt{m}\|\boldsymbol{\ve}\|_2\|\widehat{\vx}\|_2\overset{(a)}\lesssim \left|{\sum_{j=1}^m\eta_j|\langle\va_j,\widehat{\vx}\rangle|}\right|\leq \sum_{j=1}^{m}|\langle\va_j,\widehat{\vx}\rangle|\cdot\left|\langle\va_j,\widehat{\vx}-\exp(\mathrm{i}\theta)\vx_0\rangle\right|\overset{(b)}\lesssim m\|\widehat{\vx}\|_2\|\widehat{\vx}-\exp(\mathrm{i}\theta)\vx_0\|_2,
\end{equation}
for any $\theta\in [0,2\pi)$, where
 the inequalities  $(a)$ and $(b)$ follow from Lemma \ref{lem: lower} and Lemma \ref{lem: upper}, respectively.
According to (\ref{eq:bb1}),  we have
\[
\frac{\|\boldsymbol{\ve}\|_2}{\sqrt{m}}\,\,\lesssim\,\, \min_{\theta\in [0,2\pi)}\|\widehat{\vx}-\exp(\mathrm{i}\theta)\vx_0\|_2,
\]
which leads to the conclusion.

It remains to prove (\ref{eqn: subgradient_equation}). A simple observation  is that
\[
\sum_{j=1}^m\left(|\langle \va_j,\vx\rangle|-b_j\right)^2=\sum_{j=1}^m(\sqrt{(\va_{j,1}^T\vx_1+\va_{j,2}^T\vx_2)^2+(\va_{j,1}^T\vx_2-\va_{j,2}^T\vx_1)^2}-b_j)^2=:g(\vx_1,\vx_2),
\]
where $\vx_1$  and $\vx_2$ are the real and imaginary parts of $\vx$, i.e., $\vx=\vx_1+\vx_2\mathrm{i}$. Similarly, we use $\va_{j,1}$ and  $\va_{j,2}$ to denote the  real and imaginary parts of $\va_j$, i.e., $\va_j=\va_{j,1}+\va_{j,2}\mathrm{i}$.
The sub-gradient set of  $g(\vx_1,\vx_2)$ at $(\vx_1, \vx_2)$ is
\[
\begin{split}
\partial g(\vx_1, \vx_2):=\left(\begin{array}{c}
               \frac{\partial g(\vx_1, \vx_2)}{\partial \vx_1}\\
             \frac{\partial g(\vx_1, \vx_2)}{\partial \vx_2}
            \end{array}
\right)=&\Bigg\{
\left(\begin{array}{c}
              2\sum_{j=1}^m (\abs{\innerp{\va_j,\vx}}-b_j)(z_{j,1}\va_{j,1}-z_{j,2}\va_{j,2}) \\
              2\sum_{j=1}^m (\abs{\innerp{\va_j,\vx}}-b_j)(z_{j,1}\va_{j,2}+z_{j,2}\va_{j,1})
            \end{array}
\right): \\
&\qquad(z_{j,1}, z_{j,2})\in \R^2 \text{ is defined in } (\ref{eq:z1z2}) \text{ or } (\ref{eq:z1z20})
\Bigg\}.
\end{split}
\]
Here,
\begin{equation}\label{eq:z1z2}
z_{j,1}:=\frac{\va_{j,1}^T\vx_1+\va_{j,2}^T\vx_2}{\abs{\innerp{\va_j, \vx}}},  \quad
z_{j,2}:=\frac{\va_{j,1}^T\vx_2-\va_{j,2}^T\vx_1}{\abs{\innerp{\va_j, \vx}}}, \text{ if } \ \abs{\innerp{\va_j,\vx}}\neq 0.
\end{equation}
Otherwise, we require  $(z_{j,1},z_{j,2})\in \R^2 $ satisfies
\begin{equation}\label{eq:z1z20}
z_{j,1}^2+z_{j,2}^2\leq 1, \quad \text{ if }\  \abs{\innerp{\va_j,\vx}}= 0.
\end{equation}

Recall that $\widehat{\vx}=\widehat{\vx}_{1}+\widehat{\vx}_{2}\mathrm{i}$ is a solution to  (\ref{eqn: nonsparse_model}). Then we have
$\boldsymbol{0}\in \partial g(\vx_1, \vx_2)|_{\vx_1=\widehat{\vx}_1,\vx_2=\widehat{\vx}_2}$. Therefore, denote specific $\widehat{z}_{j,1}$ and $\widehat{z}_{j,2},$ $ j=1,\ldots,m$, such that
\[
\boldsymbol{0}=2\sum_{j=1}^m\left(\abs{\innerp{\va_j,\widehat{\vx}}}-b_j\right)(\widehat{z}_{j,1}\va_{j,1}-\widehat{z}_{j,2}\va_{j,2})=:\widehat{\vg}_1,
\]
and
\[
\boldsymbol{0}=2\sum_{j=1}^m\left(\abs{\innerp{\va_j,\widehat{\vx}}}-b_j\right)(\widehat{z}_{j,1}\va_{j,2}+\widehat{z}_{j,2}\va_{j,1})=:\widehat{\vg}_2.
\]
Then by direct calculation, we have
\begin{equation}
0=\langle \widehat{\vg}_1+\widehat{\vg}_2\mathrm{i},\widehat{\vx}_{1}+\widehat{\vx}_{2}\mathrm{i}\rangle=2\sum_{j=1}^{m}(|\langle\va_j,\widehat{\vx}\rangle|-b_j)|\langle\va_j,\widehat{\vx}\rangle|.
\end{equation}
\end{proof}
\section{Proof of Theorem \ref{thm: better_estimation}}
\begin{lemma}\label{lem: upper_new}
Assume that $\va_j,$  $j=1,\ldots,m\, (m\geq 3),$  are
i.i.d. complex sub-Gaussian random vectors,  whose elements are independently drawn from some sub-Gaussian random variable $\xi$ with $\|\xi\|_{\psi_2}\leq C_{\psi}$, $\mathbb{E}\xi=0$, $\mathbb{E}|\xi|^2=1$, and $\mathbb{E}|\xi|^4=\gamma>1$. Denote $\boldsymbol{\ve}=(\eta_1,\dots,\eta_m)\in \R^m$.
Then   the following holds with probability at least $1-2\exp\left(-m\right)-2\exp(-d\log m)$:
\begin{equation}\label{eqn: eta_upper}
\Bigg|\sum_{j=1}^m\eta_j\left|\langle \va_j,\vx\rangle\right|\Bigg|\lesssim \sqrt{d\log m}\cdot C_{\psi}\cdot\|\boldsymbol{\eta}\|_2+C_{\psi}\left|\sum_{j=1}^m\eta_j\right|\quad \text{for all}\  \vx\in \mathbb{S}^{d-1},
\end{equation}
provided   $m\gtrsim d$. Here, $\mathbb{S}^{d-1}$ is the unit complex sphere in $\mathbb{C}^{d-1}$.
\end{lemma}
\begin{proof}
For any fixed $\boldsymbol{x}\in \mathbb{S}^{d-1}$, by applying (\ref{eqn: vector_subGaussian}), we can deduce that   $\||\langle\va_1, \vx\rangle|\|_{\psi_2}\leq \sqrt{C'}C_{\psi}$ and $\mathbb{E}|\langle \va_1,\vx\rangle|\leq \sqrt{C'}C_{\psi}$. Taking $t=C_1\sqrt{d}\|\boldsymbol{\eta}\|_2$ in (\ref{eq:lowergauss}), we have
  \begin{equation}\label{eq:concentration_new}
 \mathbb{P}\left\{\left|\sum_{j=1}^m \eta_j|\langle \va_{j},\vx\rangle|-\sum_{j=1}^m\eta_j\mathbb{E}|\langle \va_{j},\vx\rangle|\right|\geq C_1\sqrt{d}\|\boldsymbol{\eta}\|_2\right\}\leq 2 \exp\left(-\frac{c'C_1^2d}{C'C_{\psi}^{2}}\right).
 \end{equation}
 The specific value of $C_1$ will be determined later.  Assume that $\mathcal{N}$ is an $\epsilon$-net of  $\mathbb{S}^{d-1}$ with  $\#\mathcal{N}\leq (1+\frac{2}{\epsilon})^{2d}$.
 By the union bound, we obtain that
 \begin{equation}\label{eqn: cover_uniform_new}
\left|\sum_{j=1}^m\eta_j|\langle \va_j,\vx\rangle|\right|\leq C_1\sqrt{d}\|\boldsymbol{\eta}\|_2+\left|\sum_{j=1}^m\eta_j\mathbb{E}|\langle \va_j,\vx\rangle|\right|\leq C_1\sqrt{d}\|\boldsymbol{\eta}\|_2+\sqrt{C'}C_\psi\left|\sum_{j=1}^m\eta_j\right|
 \end{equation}
 for any $\vx\in \mathcal{N}$ with probability at least $1-2 \exp\left(-\frac{c'C_1^2}{C'C_{\psi}^2}d+2\log(1+2/\epsilon)d\right)$.

 For any $\boldsymbol{x}\in \mathbb{S}^{d-1}$, there exists some $\boldsymbol{x}'\in \mathcal{N}$ such that $\|\boldsymbol{x}-\boldsymbol{x}'\|_2\leq \epsilon$. Therefore, we have
 \begin{equation}\label{eqn: eta_temp}
 \begin{split}
\left| \sum_{j=1}^m\eta_j|\langle \va_j,\vx\rangle|\right|&\leq \left|\sum_{j=1}^m\eta_j|\langle \va_j,\vx'\rangle|\right|+\sum_{j=1}^m|\eta_j|\cdot|\langle \va_j,\vx'-\vx\rangle|\\
 &\leq C_1\sqrt{d}\|\boldsymbol{\eta}\|_2+\sqrt{C'}C_\psi\left|\sum_{j=1}^m\eta_j\right|+3\cdot C_\psi\cdot C'\cdot  \sqrt{m}\|\ve\|_2\epsilon
 \end{split}
 \end{equation}
 with probability at least
 \[1-2\exp(-m)-2 \exp\left(-\frac{c'C_1^2}{C'C_{\psi}^2}d+2\log(1+2/\epsilon)d\right),\]provided that $m\gtrsim d$.  The last inequality in (\ref{eqn: eta_temp}) is based on (\ref{eqn: temp_upper0}).
 Taking $\epsilon=1/\sqrt{m}$ and $C_1^2=\frac{4C'C_\psi^2}{c'}\log(1+2\sqrt{m})$, we can immediately obtain the conclusion.
\end{proof}
\begin{proof}[Proof of Theorem \ref{thm: better_estimation}]
Since $b_j=|\langle \va_j,\vx_0\rangle|+\eta_j$, $j=1,\ldots,m$, and \[
\sum_{j=1}^m(\abs{\innerp{\va_j,\widehat{\vx}}}-b_j)^2\leq \sum_{j=1}^m(\abs{\innerp{\va_j,\vx_0}}-b_j)^2,
\]
we have
\begin{equation}\label{eqn: thm3_final}
\sum_{j=1}^m\left(|\langle \va_j,\widehat{\vx}\rangle|-|\langle \va_j,\vx_0\rangle|\right)^2\leq \sum_{j=1}^m2\eta_j\left(|\langle \va_j,\widehat{\vx}\rangle|-|\langle \va_j,\vx_0\rangle|\right).
\end{equation}
We will now estimate both sides of equation (\ref{eqn: thm3_final}) separately.
 Let's start by considering the left side of (\ref{eqn: thm3_final}).
Set $\xi_j:=\abs{\innerp{\va_j,\widehat{\vx}}}+\abs{\innerp{\va_j,\vx_0}}$ and assume that $\xi_{j_1}\leq \xi_{j_2}\leq \cdots\leq \xi_{j_m}$. Take $I=\{j_t\ |\ t\leq (1-\beta_0)m\}$, where $\beta_0>0$ is defined in Theorem \ref{thm: SRIP}. Then we have
\begin{equation}\label{eqn: thm3_temp1}
\begin{split}
\sum_{j=1}^m\left(|\langle \va_j,\widehat{\vx}\rangle|-|\langle \va_j,\vx_0\rangle|\right)^2&\geq \sum_{j\in I} \left(|\langle \va_j,\widehat{\vx}\rangle|-|\langle \va_j,\vx_0\rangle|\right)^2\\
&=\sum_{j\in I}\frac{\langle\va_j\va_j^*,\widehat{\vx}\widehat{\vx}^*-\vx_0\vx_0^*\rangle^2}{(|\langle \va_j,\widehat{\vx}\rangle|+|\langle \va_j,\vx_0\rangle|)^2}\overset{(a)}\geq \sum_{j\in I}\frac{\langle\va_j\va_j^*,\widehat{\vx}\widehat{\vx}^*-\vx_0\vx_0^*\rangle^2}{2C_{\beta_0}(\|\vx_0\|_2+\|\widehat{\vx}\|_2)^2}\\
&=\frac{\|\mathcal{A}_{I}(\widehat{\vx}\widehat{\vx}^*-\vx_0\vx_0^*)\|_2^2}{2C_{\beta_0}(\|\vx_0\|_2+\|\widehat{\vx}\|_2)^2}\geq \frac{\|\mathcal{A}_{I}(\widehat{\vx}\widehat{\vx}^*-\vx_0\vx_0^*)\|_1^2}{2C_{\beta_0}(1-\beta_0)m(\|\vx_0\|_2+\|\widehat{\vx}\|_2)^2}\\
&\overset{(b)}\geq \frac{C_{-}^2m\|\widehat{\vx}\widehat{\vx}^*-\vx_0\vx_0^*\|_F^2}{2C_{\beta_0}(1-\beta_0)(\|\vx_0\|_2+\|\widehat{\vx}\|_2)^2}\overset{(c)}\geq \frac{C_{-}^2m}{16C_{\beta_0}(1-\beta_0)}\min_{\theta\in[0,2\pi)}\|\widehat{\vx}-\exp(\mathrm{i}\theta)\vx_0\|_2^2.
\end{split}
\end{equation}
Here the inequalities $(a), (b)$  and $(c)$ are based on (\ref{eq:vxup}), (\ref{eq:sqrt2}) and Theorem \ref{thm: SRIP}, respectively.

Next, let's move on to the right side of  (\ref{eqn: thm3_final}). Based on Lemma \ref{lem: upper_new}, we have
\begin{equation}\label{eqn: thm3_temp2}
\begin{split}
\sum_{j=1}^m2\eta_j\left(|\langle \va_j,\widehat{\vx}\rangle|-|\langle \va_j,\vx_0\rangle|\right)&\leq \left|\sum_{j=1}^m2\eta_j|\langle \va_j,\widehat{\vx}\rangle|\right|+\left|\sum_{j=1}^m2\eta_j|\langle \va_j,\vx_0\rangle|\right|\\
& \leq 2C\left( \sqrt{d\log m}\cdot C_{\psi}\cdot\|\boldsymbol{\eta}\|_2+C_{\psi}\left|\sum_{j=1}^m\eta_j\right| \right)(\|\widehat{\vx}\|_2+\|{\vx}_0\|_2)
\end{split}
\end{equation}
for some absolute constant $C>0$.
Therefore, plugging (\ref{eqn: thm3_temp1}) and (\ref{eqn: thm3_temp2}) into (\ref{eqn: thm3_final}), we have
\begin{equation}\label{final_new1}
\frac{C_{-}^2m}{16C_{\beta_0}(1-\beta_0)}\min_{\theta\in[0,2\pi)}\|\widehat{\vx}-\exp(\mathrm{i}\theta)\vx_0\|_2^2\leq 2C\left( \sqrt{d\log m}\cdot C_{\psi}\cdot\|\boldsymbol{\eta}\|_2+C_{\psi}\left|\sum_{j=1}^m\eta_j\right| \right)(\|\widehat{\vx}\|_2+\|{\vx}_0\|_2).
\end{equation}
By performing a straightforward calculation, we can demonstrate that for any $a,b,l\geq 0$ satisfying $(a-b)^2\leq l(a+b)$, the following inequality holds:
\begin{equation}\label{eq:abc}
a\leq \frac{2b+l+\sqrt{(2b+l)^2-4b^2+4lb}}{2}=\frac{2b+l+\sqrt{8lb+l^2}}{2}\leq 3b+l.
\end{equation}
Set $a_0:=\|\widehat{\vx}\|_2$, $b_0:=\|{\vx}_0\|_2$ and
\begin{equation}\label{c}
l_0:=\frac{16C_{\beta_0}(1-\beta_0)}{C_{-}^2m}\cdot 2C\cdot \left(\sqrt{d\log m}\cdot C_{\psi}\cdot\|\boldsymbol{\eta}\|_2+C_{\psi}\left|\sum_{j=1}^m\eta_j\right|\right).
\end{equation}
By noting that (\ref{final_new1}) implies $(a_0-b_0)^2\leq l_0(a_0+b_0)$, we can utilize  (\ref{eq:abc}) to obtain the following result
\begin{equation}\label{eq:vxvx0}
\|\widehat{\vx}\|_2\,\,\leq\,\, 3\|\vx_0\|_2+l_0.
\end{equation}
Substituting (\ref{eq:vxvx0}) into the right hand side of (\ref{final_new1}), we obtain that
\begin{equation}\label{eqn: hat_x with l}
\min_{\theta\in[0,2\pi)}\|\widehat{\vx}-\exp(\mathrm{i}\theta)\vx_0\|_2^2\leq (4\|\vx_0\|_2+l_0)l_0,
\end{equation}
where $l_0$ is defined as in (\ref{c}).
Recall that
\begin{equation}\label{eqn: upper_eta_m}
 \left|\sum_{j=1}^m\eta_j\right|\lesssim  \sqrt{m\log m}\qquad\text{and}\qquad \sum_{j=1}^m \eta_j^2\lesssim m.
\end{equation}

Substituting (\ref{eqn: upper_eta_m}) into (\ref{eqn: hat_x with l}) and (\ref{c}), we have
\[
\begin{split}
\min_{\theta\in[0,2\pi)}\|\widehat{\vx}-\exp(\mathrm{i}\theta)\vx_0\|_2^2&\leq C'_0 (\|\vx_0\|_2+\sqrt{d\log m/m}+\sqrt{\log m/m})\cdot (\sqrt{d\log m/m}+\sqrt{\log m/m})\\
&\leq C_0 \max\left\{\|\vx_0\|_2,\sqrt{\frac{d\log m}{m}}\right\}\sqrt{\frac{d\log m}{m}},
\end{split}
\]
where $C'_0$ and $C_0$ are some positive constant only depending on $C_{\psi}$ and $\gamma=\mathbb{E}|\xi|^4$.
\end{proof}

\section{Proof of Theorem \ref{thm: sparse_result}}

Firstly, we extend Lemma \ref{lem: non_sparse} to sparse signals, which proves to be highly beneficial for our proof.

 \begin{lemma}\label{lem: sparse}
Assume  that $\vx_0\in \mathbb{C}^d$ is $k$-sparse.
Assume that $\va_j\in \mathbb{C}^d$ and $
b_j=\abs{\innerp{\va_j,\vx_0}}+\eta_j, \  j=1,\ldots,m,$ with $\|\boldsymbol{\eta}\|_2\leq \epsilon$.
We assume that $\widehat{\vx}\in \mathbb{C}^d$  satisfies $\sqrt{\sum_{j=1}^m(\abs{\innerp{\va_j,\widehat{\vx}}}-b_j)^2}\leq \epsilon$ and $\|\widehat{\vx}\|_1\leq \|\vx_0\|_1$.
 Under the assumption of Theorem \ref{th:upbound} about $\va_j\in \mathbb{C}^d, j=1,\ldots,m$,
if $m\gtrsim k\log(ed/k)$, then the following holds with probability at least $1-2 \exp(-c_0m/4)$:
\begin{equation}
\|\mathcal{A}_I({ \widehat{\vx}\widehat{\vx}^{*}-\vx_0\vx_0^*})\|_1\lesssim\sqrt{m}\epsilon(\|\widehat{\vx}\|_2+\|\vx_0\|_2).
\end{equation}
Here, $c_0$ is the positive constant defined in Theorem \ref{thm: RIP}.
The index set $I:=I_{\widehat{\vx},\vx_0}\subset\{1,\ldots,m\}$ is chosen by the following way. Set $\xi_j:=\abs{\innerp{\va_j,\widehat{\vx}}}+\abs{\innerp{\va_j,\vx_0}}$ and assume that $\xi_{j_1}\leq \xi_{j_2}\leq \cdots\leq \xi_{j_m}$. Take $I=\{j_t\ |\ t\leq (1-\beta_0)m\}$, where $\beta_0$ is an absolute constant defined in Theorem \ref{thm: SRIP}.\end{lemma}

 \begin{proof}
 Using a similar argument for   (\ref{eqn: commonly_used1}), we obtain that
\begin{equation}\label{eqn: upper_sparse_temp1}
\sum_{j\in I}\frac{\big|\left|\innerp{\va_j\va_j^*, \widehat{\vx}\widehat{\vx}^{*}-\vx_0\vx_0^*}\right|-|\eta_j|(\abs{\innerp{\va_j,\widehat{\vx}}}+\abs{\innerp{\va_j,\vx_0}})\big|}{\abs{\innerp{\va_j,\widehat{\vx}}}+\abs{\innerp{\va_j,\vx_0}}}\leq \sum_{j=1}^m \left|\abs{\innerp{\va_j,\widehat{\vx}}}-b_j\right|\leq \sqrt{m}\epsilon.
\end{equation}
We claim  that
\begin{equation}\label{eqn: upper_sparse_temp2}
\left(\abs{\innerp{\va_j,\vx_0}}+\abs{\innerp{\va_j,\widehat{\vx}}}\right)^2\leq  2C_{\beta_0}(\|\vx_0\|_2+\|\widehat{\vx}\|_2)^2,  \text{ for } j\in I,
\end{equation}
where $C_{\beta_0}:=\left(2+{4\sqrt{2}}\right)\frac{C_{+}}{\beta_0}$.
Combining (\ref{eqn: upper_sparse_temp1}) and (\ref{eqn: upper_sparse_temp2}), we can obtain that
\begin{equation}\label{eq:xishuproof}
\begin{aligned}
\frac{\left|\sum_{j\in I}\left|\innerp{\va_j\va_j^*, \widehat{\vx}\widehat{\vx}^{*}-\vx_0\vx_0^*}\right|-\sum_{j\in I}|\eta_j|(\abs{\innerp{\va_j,\widehat{\vx}}}+\abs{\innerp{\va_j,\vx_0}})\right|}{\sqrt{2C_{\beta_0}}(\|\vx_0\|_2+\|\widehat{\vx}\|_2)}\leq \sqrt{m}\epsilon.
\end{aligned}
\end{equation}
We can employ (\ref{eq:xishuproof}) to obtain that
\[
  \|\mathcal{A}_I({ \widehat{\vx}\widehat{\vx}^{*}-\vx_0\vx_0^*})\|_1\lesssim\sqrt{m}\epsilon(\|\widehat{\vx}\|_2+\|\vx_0\|_2).
\]
Here, we use the same argument as in the proof of Lemma \ref{lem: non_sparse} and we omit the details.
It remains to prove (\ref{eqn: upper_sparse_temp2}).
Set \[
\vZ:=\vX_0+\widehat{\vX}=\vx_0\vx_0^*+\widehat{\vx}\widehat{\vx}^*,\quad \text{with}\ \quad \vX_0:=\vx_0\vx_0^*\quad\ \text{and}\quad\  \widehat{\vX}:=\widehat{\vx}\widehat{\vx}^*.
\]
 Set
$S_{0}:=\text{supp}(\vx_{0})\subset \{1,\ldots,d\}$.
 Set $S_{1}$ as the index set which contains the indices of the $k$ largest
 elements of  $\widehat{\vx}_{S_{0}^{c}}$ in magnitude, and $S_2$ contains the indices of the next $k$ largest elements, and so on. For simplicity, we set $S_{01}:=S_{0}\cup S_{1}$
and  $\widetilde{\vZ}:=Z_{S_{01},S_{01}}$. Then, for any $j\in I$, we have
\begin{equation}\label{eqn: upper_sparse_final1}
\begin{aligned}
&\left(\abs{\innerp{\va_j,\vx_0}}+\abs{\innerp{\va_j,\widehat{\vx}}}\right)^2\leq {\frac{1}{\beta_0m}\sum_{t\in \{1,\ldots,m\}\setminus I } (\abs{\innerp{\va_t,\vx_0}}+\abs{\innerp{\va_t,\widehat{\vx}}})^2}\\
&\leq \frac{2}{\beta_0m}\sum_{t\in \{1,\ldots,m\}\setminus I } (\abs{\innerp{\va_t,\vx_0}}^2+\abs{\innerp{\va_t,\widehat{\vx}}}^2)=\frac{2}{\beta_0m}\|\mathcal{A}_{\{1,\ldots,m\}\setminus I}(\vZ)\|_1\\
&\leq \frac{2}{\beta_0m}\|\mathcal{A}_{\{1,\ldots,m\}\setminus I}(\widetilde{\vZ})\|_1+\frac{2}{\beta_0m}\|\mathcal{A}_{\{1,\ldots,m\}\setminus I}(\vZ-\widetilde{\vZ})\|_1
\end{aligned}
\end{equation}
 According to Theorem \ref{thm: RIP}, we have
 \begin{equation}\label{eqn: upper_sparse_temp11}
 \frac{1}{\beta_0m}\|\mathcal{A}_{\{1,\ldots,m\}\setminus I}(\widetilde{\vZ})\|_1\leq \frac{1}{\beta_0m}\|\mathcal{A}(\widetilde{\vZ})\|_1\leq \frac{C_{+}}{\beta_0}\|\widetilde{\vZ}\|_F\leq \frac{C_{+}}{\beta_0}(\|\vx_0\|_2^2+\|\widehat{\vx}\|_2^2).
 \end{equation}
 We also have
 \begin{equation}\label{eqn: mid_final0}
 \begin{aligned}
\frac{1}{\beta_0 m}\|\mathcal{A}_{\{1,\ldots,m\}\setminus I}(Z-\widetilde{\vZ})\|_1
 \leq  & \frac{1}{\beta_0m}\|\A(\vZ_{S_{0},S_{01}^c}+\vZ_{S_{01}^c,S_{0}})\|_1+\frac{1}{\beta_0m}\|\A(\vZ_{S_{1},S_{01}^c}+\vZ_{S_{01}^c,S_{1}})\|_1\\
 &+\frac{1}{\beta_0m}\|\A(\vZ_{S_{01}^c,S_{01}^c})\|_1.
 \end{aligned}
  \end{equation}
 We claim that
  \begin{equation}\label{eqn: mid_final1}
\frac{1}{\beta_0 m}\|\A(\vZ_{S_{l},S_{01}^c}+\vZ_{S_{01}^c,S_{l}})\|_1\leq \frac{2\sqrt{2}C_{+}}{\beta_0}(\|\vx_0\|_2^2+\|\widehat{\vx}\|_2^2),\quad \text{ for } l \in \{0,1\},
  \end{equation}
  and
    \begin{equation}\label{eqn: mid_final2}
 \frac{1}{\beta_0m}\|\A(\vZ_{S_{01}^c,S_{01}^c})\|_1\leq \frac{C_{+}}{\beta_0}(\|\vx_0\|_2^2+\|\widehat{\vx}\|_2^2).
  \end{equation}
  Combining (\ref{eqn: upper_sparse_final1}), (\ref{eqn: upper_sparse_temp11}), (\ref{eqn: mid_final0}), (\ref{eqn: mid_final1}) and (\ref{eqn: mid_final2}),  we obtain that
  \begin{equation}\label{eqn: upper_sparse_final2}
 \left(\abs{\innerp{\va_j,\vx_0}}+\abs{\innerp{\va_j,\widehat{\vx}}}\right)^2\leq 2\left(2+{4\sqrt{2}}\right)\frac{C_{+}}{\beta_0}(\|\vx_0\|_2^2+\|\widehat{\vx}\|_2^2).
  \end{equation}
The only remaining task is to validate (\ref{eqn: mid_final1}) and (\ref{eqn: mid_final2}). Let's first focus on (\ref{eqn: mid_final1}).
  According to  Theorem \ref{thm: SRIP}, for $l\in \{0,1\}$, we have
\[
\begin{aligned}
&\frac{1}{\beta_0 m}\|\A(\vZ_{S_{l},S_{01}^c}+\vZ_{S_{01}^c,S_{l}})\|_1\\
&\leq \sum_{j\geq 2}\frac{1}{\beta_0m}\|\A(\vZ_{S_{l},S_j}+\vZ_{S_{j},S_{l}})\|_1\leq \sum_{j\geq 2}\frac{C_{+}}{\beta_0}\|\vZ_{S_{l},S_j}+\vZ_{S_{j},S_{l}}\|_F\\
& \leq \frac{C_{+}}{\beta_0} \sum_{j\geq 2}(\|\widehat{\vx}_{S_l}\widehat{\vx}_{S_j}^*\|_F+\|\widehat{\vx}_{S_j}\widehat{\vx}_{S_l}^*\|_F)=\frac{2C_{+}}{\beta_0}\sum_{j\geq 2}\|\widehat{\vx}_{S_l}\|_2\|\widehat{\vx}_{S_j}\|_2\\
&\overset{(a)}\leq \frac{2C_{+}}{\beta_0} \|\widehat{\vx}_{S_{l}}\|_{2}\|\widehat{\vx}_{S_{01}}-\vx_0\|_{2}\leq  \frac{2C_{+}}{\beta_0}\|\widehat{\vx}_{S_{01}}\|_{2}\|\widehat{\vx}_{S_{01}}-\vx_0\|_{2}\leq  \frac{2\sqrt{2}C_{+}}{\beta_0}\|\widehat{\vx}_{S_{01}}\widehat{\vx}_{S_{01}}^*-\vx_0\vx_0^*\|_F\\
&\leq  \frac{2\sqrt{2}C_{+}}{\beta_0}(\|\vx_0\|_2^2+\|\widehat{\vx}\|_2^2),
\end{aligned}
\]
where the inequality $(a)$  is obtained by
\begin{equation}\label{eq:term31}
\begin{aligned}
\|\widehat{\vx}_{S_{l}}\|_{2} \cdot\sum_{j\geq2}\|\widehat{\vx}_{S_{j}}\|_{2}
\leq\frac{1}{\sqrt{k}}\|\widehat{\vx}_{S_{0}^{c}}\|_{1}\|\widehat{\vx}_{S_{l}}\|_{2}
\leq \|\widehat{\vx}_{S_{l}}\|_{2}\|\widehat{\vx}_{S_{01}}-\vx_0\|_{2}.
\end{aligned}
\end{equation}
 {The first inequality in (\ref{eq:term31}) follows from $\|\widehat{\vx}_{S_{j}}\|_2\leq \|\widehat{\vx}_{S_{j-1}}\|_1/\sqrt{k}$, for $j\geq 2$}, and the second inequality in (\ref{eq:term31}) is based on $\|\widehat{\vx}\|_1\leq \|\vx_0\|_1$, which leads to
\[
\|\widehat{\vx}_{S_{0}^c}\|_{1}\leq \|\vx_0\|_1-\|\widehat{\vx}_{S_{0}}\|_1\leq \|\widehat{\vx}_{S_{0}}-\vx_0\|_1\leq \sqrt{k}\|\widehat{\vx}_{S_{0}}-\vx_0\|_2\leq \sqrt{k}\|\widehat{\vx}_{S_{01}}-\vx_0\|_2.
\]
We next turn to (\ref{eqn: mid_final2}).
We have
\[
\begin{aligned}
\frac{1}{\beta_0m}\|\A(\vZ_{S_{01}^c,S_{01}^c})\|_1
 \leq &\frac{1}{\beta_0m}\sum_{l\geq j\geq2}\|\A(\vZ_{S_{l},S_{j}})+\A(\vZ_{S_{j},S_{l}})\|_1\\
 \leq& \frac{C_{+}}{\beta_0} \sum_{l\geq2,j\geq2}\|\widehat{\vx}_{S_{l}}\|_{2}\cdot \|\widehat{\vx}_{S_{j}}\|_{2}
=\frac{C_{+}}{\beta_0}  \left(\sum_{l\geq2}\|\widehat{\vx}_{S_{l}}\|_{2}\right)^{2}\\
\leq&\frac{C_{+}}{\beta_0} \cdot\frac{1}{k}\|\widehat{\vx}_{S_{0}^{c}}\|_{1}^{2}
=\frac{C_{+}}{\beta_0} \cdot\frac{1}{k}\|\vZ_{S_{0}^{c},S_{0}^{c}}\|_{1} \\
 \overset{(b)}\leq&  \frac{C_{+}}{\beta_0} \cdot\frac{1}{k}\|\vZ_{S_{0},S_{0}}\|_{1}\leq\frac{C_{+}}{\beta_0}\|\vZ_{S_{0},S_{0}}\|_{F}\leq\frac{C_{+}}{\beta_0} \|\widetilde{\vZ}\|_{F}\\
\leq&  \frac{C_{+}}{\beta_0}(\|\vx_0\|_2^2+\|\widehat{\vx}\|_2^2).
 \end{aligned}
\]
We arrive at  (\ref{eqn: mid_final2}).
The inequality $(b)$ is based on $\|\vZ_{S_{0}^{c},S_{0}^{c}}\|_{1}\leq \|\vZ-\vZ_{S_{0},S_{0}}\|_{1}\leq \|\vZ_{S_{0},S_{0}}\|_{1}$.
Indeed, according to $\|\widehat{\vX}\|_1\leq \|\vX_0\|_1$, we have
\[
\|\vZ-\vZ_{S_{0},S_{0}}\|_{1}=\|\widehat{\vX}-\widehat{\vX}_{S_0,S_0}\|_1\leq \|\vX_0\|_1-\|\widehat{\vX}_{S_0,S_0}\|_1\leq \|\vX_0+\widehat{\vX}_{S_0,S_0}\|_1=\|\vZ_{S_0,S_0}\|_1.
\]
\end{proof}

Next, we utilize Lemma \ref{lem: sparse} to prove Theorem \ref{thm: sparse_result}. The proof follows a similar scheme as in \cite[Theorem 1.3]{XiaXu}, and we provide a detailed presentation here to ensure the completeness and convenience for the readers.

  \begin{proof}[Proof of Theorem \ref{thm: sparse_result}]
  Noting that $ \|\widehat{\vx}\|_1\leq \|\vx_0\|_1 $,
 we obtain that
 $\|\widehat{\vX}\|_1\leq \|\vX_0\|_1$ where $\vX_0:=\vx_0\vx_0^*$ and $\widehat{\vX}:=\widehat{\vx}\widehat{\vx}^*$. Set $\vH:=\widehat{\vX}-\vX_0=\widehat{\vx}\widehat{\vx}^{*}-\vx_{0}\vx_{0}^{*}$ and
$T_{0}:=\text{supp}(\vx_{0})$.
 Set $T_{1}$ as the index set which contains the indices of the $a\cdot k$ largest
 elements of  $\widehat{\vx}_{T_{0}^{c}}$ in magnitude, and $T_2$ contains the indices of the next $a\cdot k$ largest elements, and so on. Take $T_{01}:=T_{0}\cup T_{1}$
and  $\widetilde{\vH}:=\vH_{T_{01},T_{01}}$ for simplicity.  {Given that $\exp(\mathrm{i}\theta)\widehat{\vx}$ is also a solution for any $\theta\in \mathbb{R}$, we can, for the purpose of applying Lemma \ref{Lemma 3.2}, assume that
\[
 \langle \widehat{\vx},\vx_0\rangle= \langle \widehat{\vx}_{T_{01}},\vx_0\rangle\geq 0.
 \]}
Here, we require that the absolute constant $a>0$ satisfies
\begin{equation}\label{eqn: a_condition}
C_{-}(1-\beta_0)-\frac{4C_{+}}{\sqrt{a}}-\frac{C_{+}}{a}>0,
\end{equation}
where $C_{-},$ $C_{+},$ and $\beta_0$ are  defined in Theorem \ref{thm: SRIP}. Set $T_{01}:=T_{0}\cup T_{1}$
and  $\widetilde{\vH}:=\vH_{T_{01},T_{01}}$.
We claim that
\begin{equation}\label{eq:mainH}
\begin{aligned}
 \|\widehat{\vx}\widehat{\vx}^* -\vx_0\vx_0^*\|_F&=\|\vH\|_F\leq \|\widetilde{\vH}\|_F+\|\vH-\widetilde{\vH}\|_F\leq \left(\frac{1}{a}+\frac{4}{\sqrt{a}}+1\right)\|\widetilde{\vH}\|_F\\
 &\leq \frac{\frac{1}{a}+\frac{4}{\sqrt{a}}+1}{C_{-}(1-\beta_0)-\frac{4C_{+}}{\sqrt{a}}-\frac{C_{+}}{a}}\frac{2 \epsilon(\|\widehat{\vx}\|_2+\|\vx_0\|_2)}{\sqrt{m}}.
 \end{aligned}
\end{equation}
 {Furthermore, based on  (\ref{eq:sqrt2}), we can derive the following inequality:}
\begin{equation}\label{eq: vector_matrix}
\frac{\sqrt{2}}{4}(\|\widehat{\vx}\|_2+\|\vx_0\|_2)\inf_{\theta\in[0,2\pi)}\|\widehat{\vx}-\exp(\mathrm{i}\theta)\vx_0\|_2
\leq \|{ \widehat{\vx}\widehat{\vx}^{*}-\vx_0\vx_0^*}\|_F.
\end{equation}
 Combining (\ref{eq:mainH}) and (\ref{eq: vector_matrix}), we obtain that
 \[
  \min_{\theta\in[0,2\pi)}\|\widehat{\vx}-\exp(\mathrm{i}\theta)\vx_0\|_2\lesssim \frac{1}{\sqrt{m}}\epsilon,
 \]
which leads to the conclusion.

We next turn to prove (\ref{eq:mainH}). The second inequality and third inequality  in (\ref{eq:mainH})
follow
 from
\begin{equation}\label{eq:main1}
\|\vH-\widetilde{\vH}\|_F\,\,\leq \,\, \left(\frac{1}{a}+\frac{4}{\sqrt{a}}\right) \|\widetilde{\vH}\|_F,
\end{equation}
and
\begin{equation}\label{eq:main2}
\|\widetilde{\vH}\|_F\leq  \frac{1}{C_{-}(1-\beta_0)-\frac{4C_{+}}{\sqrt{a}}-\frac{C_{+}}{a}}\frac{2\epsilon(\|\widehat{\vx}\|_2+\|\vx_0\|_2)}{\sqrt{m}},
\end{equation}
respectively.
To this end, it is enough to prove (\ref{eq:main1}) and (\ref{eq:main2}).


\textbf{Step 1: The proof of (\ref{eq:main1}).}

A simple observation is that
\begin{equation}\label{eq:right1}
\begin{aligned}
\|\vH-\widetilde{\vH}\|_F&\leq \sum_{l\geq 2,j\geq 2}\|\vH_{T_{l},T_j}\|_F+\sum_{l=0,1}\sum_{j\geq 2}\|\vH_{T_{l},T_j}\|_F+\sum_{j=0,1}\sum_{l\geq 2}\|\vH_{T_{l},T_j}\|_F\\
&= \sum_{l\geq 2,j\geq 2}\|\vH_{T_{l},T_j}\|_F+2\sum_{l=0,1}\sum_{j\geq 2}\|\vH_{T_{l},T_j}\|_F.
\end{aligned}
\end{equation}
On one hand, we have
\begin{equation}\label{eq:term1}
\begin{aligned}
\sum_{l\geq2,j\geq2}\|\vH_{T_{l},T_{j}}\|_{F} & = \sum_{l\geq2,j\geq2}\|\widehat{\vx}_{T_{l}}\|_{2}\cdot \|\widehat{\vx}_{T_{j}}\|_{2}
=\left(\sum_{l\geq2}\|\widehat{\vx}_{T_{l}}\|_{2}\right)^{2}\leq\frac{1}{ak}\|\widehat{\vx}_{T_{0}^{c}}\|_{1}^{2}\\
&=\frac{1}{ak}\|\vH_{T_{0}^{c},T_{0}^{c}}\|_{1}
  \leq  \frac{1}{ak}\|\vH_{T_{0},T_{0}}\|_{1}\leq\frac{1}{a}\|\vH_{T_{0},T_{0}}\|_{F}\leq\frac{1}{a}\|\widetilde{\vH}\|_{F},
 \end{aligned}
\end{equation}
according to $\|\widehat{\vx}_{T_{l}}\|_2\leq \|\widehat{\vx}_{T_{l-1}}\|_1/\sqrt{ak}$ ($l\geq 2$), and
\begin{equation}\label{eqn: H_temp}
\|\vH_{T_{0}^c,T_{0}^c}\|_{1}\leq \|\vH-\vH_{T_{0},T_{0}}\|_{1}=\|\widehat{\vX}-\widehat{\vX}_{T_0,T_0}\|_1\leq \|\vX_0\|_1-\|\widehat{\vX}_{T_0,T_0}\|_1\leq \|\vX_0-\widehat{\vX}_{T_0,T_0}\|_1=\|\vH_{T_0,T_0}\|_1.
\end{equation}
On the other hand,  for $l\in
\{0,1\}$, we have
\begin{equation}\label{eq:term3}
\begin{aligned}
\sum_{j\geq2}\|\vH_{T_{l},T_{j}}\|_{F}=\|\widehat{\vx}_{T_{l}}\|_{2} \cdot\sum_{j\geq2}\|\widehat{\vx}_{T_{j}}\|_{2}
\leq\frac{1}{\sqrt{ak}}\|\widehat{\vx}_{T_{0}^{c}}\|_{1}\cdot \|\widehat{\vx}_{T_{l}}\|_{2}
\leq \frac{1}{\sqrt{a}}\|\widehat{\vx}_{T_{l}}\|_{2}\cdot \|\widehat{\vx}_{T_{01}}-\vx_0\|_{2}.
\end{aligned}
\end{equation}
The last inequality in (\ref{eq:term3}) is based on
\[
\|\widehat{\vx}_{T_{0}^c}\|_{1}\leq \|\vx_0\|_1-\|\widehat{\vx}_{T_{0}}\|_1\leq \|\widehat{\vx}_{T_{0}}-\vx_0\|_1\leq \sqrt{k}\|\widehat{\vx}_{T_{0}}-\vx_0\|_2\leq \sqrt{k}\|\widehat{\vx}_{T_{01}}-\vx_0\|_2.
\]
 Substituting (\ref{eq:term1}) and (\ref{eq:term3}) into (\ref{eq:right1}), we can obtain that
\begin{equation}
\label{eqn: final1}
\begin{aligned}
\|\vH-\widetilde{\vH}\|_F &\leq \sum_{l\geq 2,j\geq 2}\|\vH_{T_{l},T_j}\|_F+\sum_{l=0,1}\sum_{j\geq 2}\|\vH_{T_{l},T_j}\|_F+\sum_{j=0,1}\sum_{l\geq 2}\|\vH_{T_{l},T_j}\|_F\\
&\leq\frac{1}{a}\|\widetilde{\vH}\|_{F}+\frac{2\sqrt{2}}{\sqrt{a}}
\|\widehat{\vx}_{T_{01}}\|_{2}\|\widehat{\vx}_{T_{01}}-\vx_{0}\|_{2}\leq \left(\frac{1}{a}+\frac{4}{\sqrt{a}}\right)\|\widetilde{\vH}\|_F.
\end{aligned}
\end{equation}
 {The last inequality above follows from Lemma \ref{Lemma 3.2}, i.e.,
$\|\widehat{\vx}_{T_{01}}\|_{2}\|\widehat{\vx}_{T_{01}}-\vx_{0}\|_{2}\leq \sqrt{2} \|\widehat{\vx}_{T_{01}}\widehat{\vx}_{T_{01}}^*-\vx_{0}\vx_0^* \|_F=\sqrt{2} \|\widetilde{\vH}\|_F$.
} Thus we arrive at (\ref{eq:main1}).

\textbf{Step 2: The proof of (\ref{eq:main2}).}
 According to Lemma \ref{lem: sparse},  we have
 \begin{equation}\label{eqn: sparse_main1}
 \frac{1}{m}\|\mathcal{A}_I({ \widehat{\vx}\widehat{\vx}^{*}-\vx_0\vx_0^*})\|_1\lesssim\frac{1}{\sqrt{m}}\epsilon(\|\widehat{\vx}\|_2+\|\vx_0\|_2),
 \end{equation}
 provided $m\gtrsim k\log(ed/k)$, which implies
 \begin{equation}
 \label{eqn: error estimation1}
 \frac{\epsilon(\|\widehat{\vx}\|_2+\|\vx_0\|_2)}{\sqrt{m}}\gtrsim \frac{1}{m}\|\A_{I}(\vH)\|_1\geq \frac{1}{m}\|\A_{I}(\widetilde{\vH})\|_1-\frac{1}{m}\|\A_{I}(\vH-\widetilde{\vH})\|_1.
 \end{equation}
In order to get a lower bound of $\frac{1}{m}\|\A_{I}(\widetilde{\vH})\|_1-\frac{1}{m}\|\A_{I}(\vH-\widetilde{\vH})\|_1$,  we bound $\frac{1}{m}\|\A_{I}(\widetilde{\vH})\|_1$ from below and  $\frac{1}{m}\|\A_{I}(\vH-\widetilde{\vH})\|_1$ from above. As $\text{rank}(\widetilde{\vH})\leq 2$ and $\|\widetilde{\vH}\|_{0,2}\leq (a+1)k$, we use  Theorem \ref{thm: SRIP} on the order of  $(2, (a+1)k)$ to obtain that
\begin{equation}\label{eqn: lower_final1}
\frac{1}{m}\|\mathcal{A}_I(\widetilde{\vH})\|_1\geq C_{-}(1-\beta_0)\|\widetilde{\vH}\|_F.
\end{equation}
Since $\vH-\widetilde{\vH}=(\vH_{T_{0},T_{01}^c}+\vH_{T_{01}^c,T_{0}})+(\vH_{T_{1},T_{01}^c}+\vH_{T_{01}^c,T_{1}})+\vH_{T_{01}^c,T_{01}^c},$
we have
\begin{equation}\label{mid: final}
\frac{1}{m}\|\A_{I}(\vH-\widetilde{\vH})\|_1\leq \frac{1}{m}\|\A_{I}(\vH_{T_{0},T_{01}^c}+\vH_{T_{01}^c,T_{0}})\|_1+\frac{1}{m}\|\A_{I}(\vH_{T_{1},T_{01}^c}+\vH_{T_{01}^c,T_{1}})\|_1+\frac{1}{m}\|\A_{I}(\vH_{T_{01}^c,T_{01}^c})\|_1.
\end{equation}
 According to Theorem \ref{thm: SRIP} on the order of  $(2, 2ak)$, for $l\in \{0,1\}$, we have
\begin{equation}\label{eqn: mid1}
\begin{split}
\frac{1}{m}\|\A_{I}(\vH_{T_{l},T_{01}^c}+\vH_{T_{01}^c,T_{l}})\|_1&\leq \sum_{j\geq 2}\frac{1}{m}\|\A_{I}(\vH_{T_{l},T_j}+\vH_{T_{j},T_{l}})\|_1\leq \sum_{j\geq 2}C_{+}\|\vH_{T_{l},T_j}+\vH_{T_{j},T_{l}}\|_F\\
& \leq C_{+} \sum_{j\geq 2}(\|\widehat{\vx}_{T_{l}}\widehat{\vx}_{T_{j}}^*\|_F+\|\widehat{\vx}_{T_{j}}\widehat{\vx}_{T_{l}}^*\|_F)=2C_{+}\|\widehat{\vx}_{T_{l}}\|_2\sum_{j\geq 2}\|\widehat{\vx}_{T_{j}}\|_2\\
&\leq \frac{2C_{+}}{\sqrt{a}} \|\widehat{\vx}_{T_{l}}\|_{2}\cdot \|\widehat{\vx}_{T_{01}}-\vx_0\|_{2},
\end{split}
\end{equation}
where the third line above is obtained as in (\ref{eq:term3}).

To bound  $\frac{1}{m}\|\A_{I}(\vH_{T_{01}^c,T_{01}^c})\|_1$, we have
\begin{equation}\label{eqn: mid2}
\begin{aligned}
 \frac{1}{m}\|\A_{I}(\vH_{T_{01}^c,T_{01}^c})\|_1
 \leq &\frac{1}{m}\sum_{l\geq j\geq2}\|\A_{I}(\vH_{T_{l},T_{j}})+\A_{I}(\vH_{T_{j},T_{l}})\|_1\\
 \leq &{C_{+}}\sum_{l\geq2,j\geq2}\|\widehat{\vx}_{T_{l}}\|_{2}\cdot \|\widehat{\vx}_{T_{j}}\|_{2}
={C_{+}}  \left(\sum_{l\geq2}\|\widehat{\vx}_{T_{l}}\|_{2}\right)^{2}\leq\frac{{C_{+}}}{ak}\|\widehat{\vx}_{T_{0}^{c}}\|_{1}^{2}\\
=&\frac{{C_{+}}}{ak}\|\vH_{T_{0}^{c},T_{0}^{c}}\|_{1} \leq  \frac{{C_{+}} }{ak}\|\vH_{T_{0},T_{0}}\|_{1}\leq\frac{{C_{+}}}{a}\|\vH_{T_{0},T_{0}}\|_{F} \leq\frac{{C_{+}}}{a}\|\widetilde{\vH}\|_{F}.
 \end{aligned}
\end{equation}
The third line above is based on (\ref{eqn: H_temp}). Now combining  (\ref{eqn:
mid1}) and (\ref{eqn: mid2}),  we obtain that
\begin{equation}
\label{eqn: lower_final2}
\begin{aligned}
\frac{1}{m}\|\A_{I}(\vH-\widetilde{\vH})\|_1&\leq \frac{1}{m}
\left(\|\A_{I}(\vH_{T_{0},T_{01}^c}+\vH_{T_{01}^c,T_{0}})\|_1+\|\A_{I}(\vH_{T_{1},T_{01}^c}+\vH_{T_{01}^c,T_{1}})\|_1
+\|\A_{I}(\vH_{T_{01}^c,T_{01}^c})\|_1\right)\\
&\leq \frac{2C_{+}}{\sqrt{a}} \|\widehat{\vx}_{T_{0}}\|_{2}\|\widehat{\vx}_{T_{01}}-\vx_0\|_{2}+\frac{2C_{+}}{\sqrt{a}} \|\widehat{\vx}_{T_{1}}\|_{2}\|\widehat{\vx}_{T_{01}}-\vx_0\|_{2}+\frac{C_{+}}{a}\|\widetilde{\vH}\|_F\\
&\leq \frac{2\sqrt{2}C_{+}}{\sqrt{a}} \|\widehat{\vx}_{T_{01}}\|_{2}\|\widehat{\vx}_{T_{01}}-\vx_0\|_{2}+\frac{C_{+}}{a}{\|\widetilde{\vH}\|_F}\\
&\leq C_{+}\left(\frac{4}{\sqrt{a}}+\frac{1}{a}\right)\|\widetilde{\vH}\|_F.
\end{aligned}
\end{equation}
 {The last inequality above follows from Lemma \ref{Lemma 3.2}, i.e.,
$\|\widehat{\vx}_{T_{01}}\|_{2}\|\widehat{\vx}_{T_{01}}-\vx_{0}\|_{2}\leq \sqrt{2} \|\widehat{\vx}_{T_{01}}\widehat{\vx}_{T_{01}}^*-\vx_{0}\vx_0^* \|_F=\sqrt{2} \|\widetilde{\vH}\|_F$.
} Putting  (\ref{eqn: lower_final1}) and  (\ref{eqn: lower_final2}) into (\ref{eqn: error
estimation1}), we obtain that
 \[
 \begin{aligned}
 \frac{{2}\epsilon(\|\widehat{\vx}\|_2+\|\vx_0\|_2)}{\sqrt{m}}&\geq \frac{1}{m}\|\A_{I}(\widetilde{\vH})\|_1-\frac{1}{m}\|\A_{I}(\vH-\widetilde{\vH})\|_1\\
 &\geq C_{-}(1-\beta_0)\|\widetilde{\vH}\|_F-C_{+}\left(\frac{4}{\sqrt{a}}+\frac{1}{a}\right)\|\widetilde{\vH}\|_F\\
 &=\left(C_{-}(1-\beta_0)-\frac{4C_{+}}{\sqrt{a}}-\frac{C_{+}}{a}\right)\|\widetilde{\vH}\|_F,
\end{aligned}
 \]
which implies  (\ref{eq:main2}).
 \end{proof}
 \section{Discussion}
 {{In this paper, we have examined the performance of amplitude-based models   (\ref{eqn: nonsparse_model}) and (\ref{eqn: sparse_model}) for complex phase retrieval problems.
  We have utilized a strong version of  restricted isometry property
  for an operator on the space of simultaneous low-rank and sparse matrices
  to demonstrate that the reconstruction error is on the order of $\mathcal{O}(\|\ve\|_2/\sqrt{m})$, assuming that the measurement vectors $\va_j$, $j=1,\ldots,m$ are independent complex sub-Gaussian random vectors.
    Furthermore, we have proven that the error bound is sharp.
   However, there are still numerous research questions that warrant further investigation. For instance, it would be worthwhile to explore whether other more general types of sampling matrices, such as partial circulant and random Fourier measurements, also adhere to this restricted isometry property. }}

\vspace{0.1cm}
{\small \bf Funding.}
{\small Yu Xia is supported by the National Nature Science Foundation of China   (12271133, U21A20426,11901143) and the Zhejiang Provincial Natural Science Foundation (LZ23A010002). Zhiqiang Xu is supported by the National Science Fund for Distinguished Young Scholars (12025108) and the National Nature Science Foundation of China  (12021001, 12288201).
}


\begin{thebibliography}{10}







\bibitem{generic1}
R. Balan, P. Casazza, and D. Edidin.
\newblock On signal reconstruction without phase.
\newblock {\em Applied and Computational Harmonic Analysis}, 20 (3): 345-356, 2006.

\bibitem{SimpleProof}
R. Baraniuk, M. Davenport, R. DeVore and M. Wakin.
\newblock A simple proof of the restricted isometry property for random matrices.
\newblock {\em Constructive Approximation}, 28(3): 253-263, 2008.


\bibitem{T. Cai}
 T. T. Cai, X. Li, and Z. Ma.  Optimal rates of convergence for noisy sparse phase retrieval via thresholded Wirtinger flow. The Annals of Statistics, 44 (5): 2221-2251, 2016.

\bibitem{Phaselift1} E. J. Cand\`es and X. Li.
  \newblock Solving quadratic equations via PhaseLift when there are about as many equations as unknowns.
  \newblock {\em Foundations of Computational Mathematics}, 14 (5): 1017-1026, 2014.

\bibitem{WF} E. J. Cand\`es, X. Li, and M. Soltanolkotabi.
\newblock Phase retrieval via Wirtinger flow: Theory and algorithms.
\newblock {\em IEEE Transactions on Information Theory}, 61 (5): 1985-2007, 2015.

\bibitem{Phaselift2} E. J. Cand\`es, T. Strohmer, and V. Voroninski.
 \newblock Phaselift: Exact and stable signal recovery from magnitude measurements via convex programming.
 \newblock {\em Communications on Pure and Applied Mathematics},  66 (8): 1241-1274, 2013.



\bibitem{subgaussian} Junren Chen and Michael K. Ng, Error bound of empirical $\ell_2$ risk minimization for noisy standard and generalized phase retrieval problems, arXiv preprint arXiv:2205.13827.

\bibitem{TWF} Y. Chen and E. J. Cand\`es.
\newblock Solving random quadratic systems of equations is nearly as easy as solving linear systems.
\newblock {\em Communications on Pure and Applied Mathematics}, 70 (5): 822-883, 2017.

\bibitem{CCG15} Y. Chen, Y. Chi and A. J. Goldsmith.
\newblock Exact and Stable Covariance Estimation from Quadratic Sampling via Convex Programming.
\newblock {\em IEEE Transactions on Information Theory}, 61 (7): 4034-4059, 2015.

\bibitem{generic3}
A. Conca, D. Edidin, M. Hering, and C. Vinzant.
 \newblock An algebraic characterization of injectivity in phase retrieval.
 \newblock {\em Applied and Computational Harmonic Analysis}, 38 (2): 346-356, 2015.

\bibitem{app2}
J. V. Corbett.
 \newblock The Pauli problem, state reconstruction and quantum-real numbers.
   \newblock {\em Reports on mathematical physics}, 57 (1): 53-68, 2006.

\bibitem{Unser23} Jonathan Dong, Lorenzo Valzania, Antoine Maillard, Thanh-an Pham, Sylvain Gigan, and Michael Unser.
\newblock Phase retrieval: From computational imaging to machine learning: A tutorial.
\newblock{ IEEE Signal Processing Magazine}, 40(1): 45-57, 2023.

\bibitem{albert}Albert Fannjiang and Thomas Strohmer.
\newblock The numerics of phase retrieval.
\newblock {\em Acta Numerica},  29: 125-228, 2020.


\bibitem{Fienup1}
J. R. Fienup.
\newblock Phase retrieval algorithms: a comparison.
\newblock {\em Applied optics}, 21 (15): 2758-2769, 1982.

\bibitem{Feuillen2020}
T. Feuillen, M. E. Davies and L. Jacques.
 \newblock $(\ell_1,\ell_2)$-RIP and projected back-projection reconstruction for phase-only measurements.
 \newblock {\em IEEE Signal Processing Letters}, 27: 396-400, 2020.

\bibitem{PAF} B. Gao, X. Sun, Y. Wang, and Z. Xu.
\newblock Perturbed amplitude flow for phase retrieval.
\newblock{\em IEEE Transactions on Signal Processing}, 68: 5427-5440, 2020.

\bibitem{sparse2}
B. Gao, Y. Wang and Z. Xu.
\newblock Stable Signal Recovery from Phaseless Measurements.
\newblock {\em Journal of Fourier Analysis and Applications}, 22(4): 7887-808, 2016.

\bibitem{HX21}
M. Huang and Z. Xu.
\newblock Uniqueness and stability for the solution of a nonlinear least squares problem.
\newblock {\em arXiv:2104.10841}, 2021.


\bibitem{Krahmer_subgaussian}
F. Krahmer and D. Stoger.
 \newblock Complex phase retrieval from subgaussian measurements.
 \newblock {\em Journal of Fourier Analysis and Applications}, 26(6): 27, 2020.



\bibitem{Fienup2}
R. W. Gerchberg and W. O. Saxton.
\newblock A practical algorithm for the determination of phase from image and diffraction plane pictures.
\newblock {\em Optik}, 35 (2): 237-250, 1972.


 \bibitem{noise} M. M. Ghazi, and H. Erdogan.
\newblock Image noise level estimation based on higher-order statistics.
\newblock{Multimedia Tools and Applications}, 76(2): 2379-2397, 2017.



\bibitem{HSX} Meng Huang, Shixiang Sun, and Zhiqiang Xu.
\newblock Affine phase retrieval for sparse signals via l1 minimization.
\newblock {\em Journal of Fourier Analysis and Applications}, 29-36, 2023.

\bibitem{Real_nonlinear_LS} M. Huang, and Z. Xu.
\newblock The estimation performance of nonlinear least squares for phase retrieval.
\newblock {\em IEEE Transactions on Information Theory}, 66 (12): 7967-7977, 2020.

\bibitem{forthmodel_PR} M. Huang and Z. Xu.
\newblock Performance bounds of the intensity-based estimators for noisy phase retrieval.
 \newblock   {\em arXiv preprint arXiv:2004.08764},2020.

\bibitem{ReLU1} M. Huang and Z. Xu.
\newblock Uniqueness and stability for the solution of a nonlinear least squares problem.
\newblock {\em arXiv preprint arXiv:2104.10841}, 2021.

\bibitem{LM}B. Laurent and P. Massart.
\newblock Adaptive estimation of a quadratic functional by model selection.
\newblock{The Annals of Statistics}, 28: 1302-1338, 2000.

\bibitem{sparse1}
M. Moravec, J. Romberg, and R. Baraniuk.
\newblock Compressive phase retrieval.
\newblock In{ \em  Proceedings of SPIE, International Society for Optics and Photonics}, 2007.


\bibitem{geometry_smooth} C. Ma, K. Wang, Y. Chi, and Y. Chen.
 \newblock Implicit regularization in nonconvex statistical estimation: Gradient descent converges linearly for phase retrieval, matrix completion and blind deconvolution.
 \newblock  {\em Foundations of Computational Mathematics}, 20 (3): 451-632, 2019.



 \bibitem{M} Christopher A Metzler et al.
\newblock Coherent inverse scattering via transmission matrices: Efficient phase retrieval algorithms and a public dataset.
\newblock{ In: 2017 IEEE International Conference
on Computational Photography (ICCP)}, 1-16, 2017.

\bibitem{app1}
J. Miao, P. Charalambous, J. Kirz, and D. Sayre.
\newblock Extending the methodology of x-ray crystallography to allow imaging of micrometer-sized non-crystalline specimens.
\newblock {\em Nature}, 400 (6742): 342-344, 1999.


\bibitem{AM15}
 P. Netrapalli, P.  Jain, and S. Sanghavi.
 \newblock Phase retrieval using alternating minimization.
 \newblock {\em IEEE Transactions on Signal Processing}, 63 (18): 4814-4826, 2015.


\bibitem{RP}
M. Raffaele and R. Pierri.
\newblock Performance of phase retrieval via phaselift and quadratic inversion in circular scanning case.
 \newblock {\em IEEE Transactions on Antennas and Propagation}, 67(12): 7528-7537, 2019.

 \bibitem{SQW}
 J. Sun, Q. Qu, and John Wright.
 \newblock A geometric analysis of phase retrieval.
 \newblock{\em Foundations of Computational Mathematics}, 18:1131-1198, 2018.

\bibitem{Vershynin}
R. Vershynin.
 \newblock High-dimensional probability: An introduction with applications in data science.
 \newblock U.K.: Cambridge University Press, 2018.


\bibitem{Fienup_converge}
I. Waldspurger.
 \newblock Phase retrieval with random Gaussian sensing vectors by alternating projections.
 \newblock {\em IEEE Transactions on Information Theory}, 64 (5): 3301-3312, 2018.

\bibitem{TAF}
G. Wang, G. B. Giannakis, and Y. C. Eldar.
\newblock Solving systems of random quadratic equations via truncated amplitude flow.
\newblock {\em IEEE Transactions on Information Theory}, 64 (2): 773-794, 2018.

\bibitem{sparse3}
Y. Wang and Z. Xu.
\newblock Phase retrieval for sparse signals.
\newblock {\em Applied and Computational Harmonic Analysis}, 37 (3): 531-544, 2014.


\bibitem{generic2}
Y. Wang, and Z. Xu.
\newblock Generalized phase retrieval: Measurement number, matrix recovery and beyond.
 \newblock {\em Applied and Computational Harmonic Analysis}, 47 (2): 423-446, 2017.

\bibitem{XiaXu}
Y. Xia and Z. Xu.
\newblock The recovery of complex sparse signals from few phaseless measurements.
  \newblock {\em Applied and Computational Harmonic Analysis}, 50: 1-15, 2021.

\bibitem{ReLU2}
Guanshuo Xu, Han-Zhou Wu, and Yun-Qing Shi.
\newblock Structural design of convolutional neural networks
for steganalysis.
\newblock {\em IEEE Signal Processing Letters}, 23(5):708-712, 2016.

\bibitem{YZT15}
L.-H. Yeh, J. Dong, J. Zhong, L. Tian, M. Chen, G. Tang, M. Soltanolkotabi, and L. Waller.
 \newblock Experimental robustness of Fourier ptychography phase retrieval algorithms.
  \newblock {\em Optics Express}, 23(26):  33214-33240, 2015.

\bibitem{AF} H. Zhang, and Y. Liang.
\newblock Reshaped wirtinger flow for solving quadratic system of equations.
\newblock {\em Advances in Neural Information Processing Systems}, 2622-2630, 2016.

\bibitem{RWF} H. Zhang, Y. Zhou, Y. Liang, and Y. Chi.
\newblock A nonconvex approach for phase retrieval: Reshaped wirtinger flow and incremental algorithms.
\newblock{Journal of Machine Learning Research}, 18(1): 5164-5198, 2017.



\end{thebibliography}

\end{document}